\documentclass[12pt]{article}
\usepackage[centertags]{amsmath}
\usepackage{amsfonts}
\usepackage{amssymb}
\usepackage{latexsym}
\usepackage{amsthm}
\usepackage{newlfont}
\usepackage{graphicx}
\usepackage{listings}
\usepackage{booktabs}
 \usepackage{abstract}
\newlength{\defbaselineskip}
\newcommand{\setlinespacing}[1]%
           {\setlength{\baselineskip}{#1 \defbaselineskip}}

\newcommand{\IR}{\mathbb{R}}
\newcommand{\IC}{\mathbb{C}}
\newcommand{\IZ}{\mathbb{Z}}

\newcommand{\IN}{\mathbb{N}}

\def \<{\langle}
\def\>{\rangle}

\newtheorem{Theorem}{Theorem}[section]
\newtheorem{Lemma}{Lemma}[section]
\newtheorem{Definition}{Definition}[section]
\newtheorem{prop}{Proposition}[section]

\newtheorem{Corollary}{Corollary}[section]

\numberwithin{equation}{section}

\begin{document}

\title{{ Wavelets in weighted norm spaces}
\thanks{\it Math Subject Classifications.
primary:  41A65; secondary: 41A25, 41A46, 46B20.}}
\author{K. S. Kazarian, \thanks{ Dept. of Mathematics, Mod. 17, Universidad Aut\'onoma de Madrid, 28049, Madrid, Spain e-mail: kazaros.kazarian@uam.es }
S. S. Kazaryan   \thanks{Institute of Mathematics
Armenian Academy of Sciences
375019 Yerevan
Armenia e-mail: kazarian@instmath.sci.am }
and A. San Antol\'{\i}n \thanks{ Dept. of Mathematical Analysis, Universidad de Alicante, 03690 Alicante, Spain)  e-mail: angel.sanantolin@uam.es}}
\maketitle

\begin{abstract} {We give a complete characterization of the classes of weight functions for which the Haar wavelet system for $m$-dilations, $m= 2,3,\ldots$ is an unconditional basis in $L^p(\mathbb{R},w)$. Particulary it follows that higher rank Haar wavelets are unconditional bases in the weighted norm spaces $L^p(\mathbb{R},w)$, where $w(x) = |x|^{r}, r>p-1$. These weights can have very strong zeros at the origin. Which shows that the class of weight functions for which higher rank Haar wavelets are unconditional bases is much richer than it was supposed. One of main purposes of our study is to show that weights with strong zeros should be considered if somebody is studying basis properties of a given  wavelet system in a weighted norm space.
 } \end{abstract}

\section{Introduction}

\

The wavelet analysis, since its creation, has been  used in many areas of
applied mathematics. The main idea is as simple as to find  a function ({\it wavelet or wavelet function}) defined in $\IR$ or in $\IR^{d}$ so that the system of its  dilations and translations constitute a complete orthonormal system (ONS) in $L^2(\mathbb{R})$ or in $L^2(\mathbb{R}^d)$. In fact this idea was used by A. Haar for constructing  a complete ONS in $L^2([0,1])$ such that the expansion with respect  the system of any continuous function on $[0,1]$  converges   uniformly. In the univariate case the simplest di\-lation  is the dyadic dilation. In this case a function $g\in L^2(\mathbb{R})$ is  a  wavelet  if  $\{g_{k,j}: k,j\in \IZ\},$ where $g_{k,j}(x):= 2^{k/2}g(2^{k}x-j)$ is a complete ONS in the space $L^2(\mathbb{R})$.

If we want to study the class of weight functions $w\geq 0$ for which a given wavelet  system $\{g_{k,j}: k,j\in \IZ\}$ is a basis in a certain sense in the weighted norm space $L^p(\mathbb{R},w), 1\leq p <\infty$ then we have to consider a new phenomenon which  does not arise in the case of a finite interval. In  recent years several papers were published where the above question was studied. Unfortunately in the majority of those papers the phenomenon which we are going to describe is not considered. To give a preliminary  idea about the main subject of our study   suppose that for a given weight function $w$ there exist some functions $h$ such that
\begin{enumerate}
\item[(a)] \qquad $\int_{\IR} g_{k,j}(x) h(x) dx = 0 \quad \mbox{for all}\quad k,j\in \IZ$;
\item[(b)] \qquad  $\frac{h}{w} \in L^{p'}(\mathbb{R},w),\qquad \qquad 1/p +1/p' =1$.
\end{enumerate}
Then any non trivial function $\frac{h}{w}$ as linear continuous functional will be a non trivial element in the dual space $L^{p'}(\mathbb{R},w)$. Moreover, this functional  vanishes on all $g_{k,j}: k,j\in \IZ$ .
Hence our wavelet system is not complete in the space $L^p(\mathbb{R},w)$. A necessary condition for completeness of the system $\{g_{k,j}: k,j\in \IZ\}$   in the space $L^p(\mathbb{R},w)$ is the following condition: $\frac{h}{w} \notin L^{p'}(\mathbb{R},w)$. Thus if we are going to describe all weight functions $w$ for which  the system $\{g_{k,j}: k,j\in \IZ\}$ is a basis in a certain sense in  $L^p(\mathbb{R},w)$, we have to consider all those weight functions $w$ for which the condition $(b)$ is not true. For our study
 we will  use the technique by which   similar questions were studied for incomplete systems  in the weighted norm spaces (see \cite{K:0},\cite{K:1}, \cite{K:2}). We also will show that the conditions $(a)$ and $(b)$ are not hypothetical cases. We tried to produce a readable text of our study. The authors are conscious that some results of the present paper can be proved by other methods. In those cases we have opted for more classical tools with the hope to present a text which will be accessible to a wider range of readers.

 Usually in concrete examples one has  that the wavelet $g\in L^1(\mathbb{R})\bigcap L^2(\mathbb{R}).$ Moreover, if we are going to study the wavelet system in $L^p(\mathbb{R},w)$ it will be natural to suppose that
\[ g\in L^1(\mathbb{R})\bigcap L^{\max(p,p')}(\mathbb{R}). \]
On the other hand the purpose of the present paper is not to obtain the most general results. Hence, instead of the last restriction we will suppose  that $g\in L^1(\mathbb{R})\bigcap L^{\infty}(\mathbb{R}).$ It is well known that if a wavelet $g\in L^1(\mathbb{R})\bigcap L^2(\mathbb{R})$ then its Fourier transform $\widehat{g}$ should vanish at the origin and thus the constant functions satisfy to condition $(a)$. Which means that if someone has studied the formulated question without describing the class of functions for which  the condition (a) holds and without considering weighted norm spaces $L^p(\mathbb{R},w)$ with weights $w$  which does not satisfy the condition (b), then his proof  is not complete. In \cite{GK:1}(see also \cite{GK:0}) was given a complete characterization  of weight functions $w$ for which the Haar wavelet system is an unconditional basis in $L^p(\mathbb{R},w), 1\leq p <\infty$. It was given without detailed proof because the technical details were similar with the proof given in \cite{K:2}. Almost at the same time was published the paper \cite{L:1}. In the last paper the described phenomenon was not considered. It is understandable that we will not cite all papers where similar questions have been studied.

In order to show that the question under consideration is not a technical problem which has only theoretical interest,
we  characterize the classes of weight functions for which the Haar wavelet system for $m$-dilations, $m= 2,3,\ldots$ is an unconditional basis in $L^p(\mathbb{R},w)$. From the corollary of the main theorem of the last section it follows that higher rank Haar wavelets are unconditional bases in the weighted norm spaces $L^p(\mathbb{R},w)$, where $w(x) = |x|^{r}, r>p-1$. Which shows that the class of weight functions for which higher rank Haar wavelets are unconditional bases is much richer than it was supposed  (see for example \cite{Ko:01}, where
 this question was studied). In Section \ref{sec:2} we prove an inequality for the orthogonal wavelet systems which particularly shows that in the case of general orthogonal wavelet systems the set of nontrivial functions $h$ for which the condition (a) holds is not empty. In Section \ref{sec:3} we give all preliminary results which will be used for our study. Next two sections are dedicated to the $m$th rank Haar system on $[0,1]$. The results of these sections have certain interest. It should be mentioned that they are used   for proofs given in the last section, where the main results for the higher rank Haar wavelets are obtained.

\

\section{An inequality for wavelet type systems}\label{sec:2}

If $w\geq 0$ be a  weight function on $\IR$, i.e. a non negative locally integrable function  then we write $\phi \in L^{p}(\IR,w), 1 \leq p <\infty$ if $\phi:\IR \rightarrow \IC$ is measurable on $\IR$ and the norm is defined by
\[
 \| \phi \|_{L^{p}(\IR,w)}:= \left( \int_{\IR} |\phi(t)|^{p} w(t) dt\right)^{\frac{1}{p}} < +\infty.
 \]

For a $g\in L^2(\mathbb{R})$ and $m= 2,3,\ldots$ we will denote
\begin{equation}\label{eq:1.0}
g_{k,j,m}(x):= m^{k/2}g(m^{k}x-j), \qquad k,j \in \IZ.
\end{equation}
In this paper we will use a slightly modified version of the classical definition of the Fourier transform.
For a function $f\in L^1(\mathbb{R}) \cap
 L^2(\mathbb{R})$ we put
 $$ \widehat{f}({y}) = \int_{\mathbb{R}} f ({x}) e^{-2\pi i{x} {y}} dx.$$
The characteristic function of a set $E$ is denoted by $\chi_E$ and $\IN_{0} = \IN\bigcup\{0\}$.

 The following lemma is a well known result
(cf.~\cite{D:92}, p. 132; ~\cite{HW:96}, p. 71).

\begin{Lemma}\label{lem:1}
The system $\{ h(\cdot -j): { j}\in
\IZ\},$ where $h\in L^2(\IR),$  is an orthonormal system if
and only if
\[
\sum_{{ j}\in \IZ}|\widehat{h}({ t}+  { j})|^2
 = 1 \qquad \mbox{for a.e.}\quad  { t}\in \IR  .
\]
\end{Lemma}
As an obvious corollary of the above lemma we have that if $g\in L^2(\IR)$ is a wavelet then $|\widehat{g}(t)| \leq 1$ a.e. on $\IR$.

\begin{Theorem}\label{thm:1} Let $g\in L^2(\mathbb{R})$ and $m= 2,3,\dots$. Suppose that the system $\{g_{k,j,m}: k\in \IN_{0}, j\in \IZ \}$  is
 orthonormal. Then
 \begin{equation}\label{eq:1.1}
 \sum_{k=0}^{\infty} |\widehat{g}(m^{-k}x)|^{2} \leq 1.
 \end{equation}
 \end{Theorem}
\begin{proof}
It is easy to check  that
 \begin{equation}\label{eq:1.2}
\widehat{g_{k,j,m}}(y) = m^{-k/2} \widehat{g}(m^{-k}y) e^{-2\pi ijm^{-k}y }.
 \end{equation}
It is well known that for any interval $I\subset \IR, |I|=1$ the trigonometric system $ \{e^{-2\pi ij y}\}_{j\in \IZ} $ is a complete orthonormal system in $L^{2}(I)$. Hence, for any $k\in \IZ$ and $\Delta \subset \IR, |\Delta|=m^{k}$ the system $ \{m^{-k/2} e^{-2\pi ij m^{-k} y}\}_{j\in \IZ}$ will be
a complete orthonormal system in $L^{2}(\Delta)$.  Thus for any $f\in L^2(\mathbb{R})$ such that $\mbox{supp}\widehat{f} \subseteq I, |I|=1$ we will have that
\[
S_{0}(f,x) = \sum_{j\in \IZ} \int_{\IR} f(t) \overline{g_{0,j,m}(t)} dt g_{0,j,m}(x) = \sum_{j\in \IZ} \int_{I} \widehat{f}(t)\overline{\widehat{g}(t)} e^{2\pi ijt} dt\, g_{0,j,m}(x).
\]
Which yields
\[
\widehat{S_{0}(f,\cdot)}(y)  = \sum_{j\in \IZ} \int_{I} \widehat{f}(t)\overline{\widehat{g}(t)} e^{2\pi ijt} dt\, \widehat{g}(y) e^{-2\pi ijy } = \widehat{f}(y) |\widehat{g}(y)|^{2}
\]
if $y\in I$. It should be observed that the last equality holds because $\widehat{f}(\cdot)\overline{\widehat{g}(\cdot)} \in L^{2}(I)$ which is true because of Lemma \ref{lem:1}.
If  for any $k\in \IN$  we put
\[
S_{k}(f,x) = \sum_{j\in \IZ} \int_{\IR} f(t) \overline{g_{k,j,m}(t)} dt g_{k,j,m}(x)
\]
in a same way we obtain that
\begin{equation}\label{eq:1.3}
\widehat{S_{k}(f,\cdot)}(y)  = \widehat{f}(y) |\widehat{g}(m^{-k}y)|^{2} \qquad \mbox{if} \quad y\in I.
\end{equation}

By the orthogonality of the system $\{g_{k,j}: k\in \IN_{0}, j\in \IZ \}$ we have that
\[
S_{k}(f,\cdot)\bot\, S_{k'}(f,\cdot)\qquad \mbox{in}\quad L^2(\mathbb{R})\qquad \mbox{if}\quad  k\neq k'.
\]
Hence,
\[
\int_{\IR} |\sum_{k=0}^{l} S_{k}(f,x)|^{2} dx = \sum_{k=0}^{l} \int_{\IR} |S_{k}(f,x)|^{2} dx = \sum_{k=0}^{l} \int_{\IR} |\widehat{S_{k}(f,\cdot)}(y)|^{2}dy
\]
\[
= \int_{\IR}  \sum_{k=0}^{l}  |\widehat{S_{k}(f,\cdot)}(y)|^{2}dy = \int_{\IR}|\widehat{f}(y)|^{2}\bigg| \sum_{k=0}^{l} |\widehat{g}(m^{-k}y)|^{2}  \bigg|^{2} dy.
\]
By (\ref{eq:1.3}) we have that
\[
\int_{I} |\widehat{f}(y)|^{2} \bigg|\sum_{k=0}^{l}  |\widehat{g}(m^{-k}y)|^{2}  \bigg|^{2} dy = \int_{I}  \sum_{k=0}^{l}  |\widehat{S_{k}(f,\cdot)}(y)|^{2}dy
\]
\[
\leq \int_{\IR}  \sum_{k=0}^{l}  |\widehat{S_{k}(f,\cdot)}(y)|^{2}dy.
\]
By the above relations and Bessel's inequality we obtain that
\[
\int_{I} |\widehat{f}(y)|^{2} \bigg|\sum_{k=0}^{l}  |\widehat{g}(m^{-k}y)|^{2}  \bigg|^{2} dy \leq \int_{\IR} |\sum_{k=0}^{l} S_{k}(f,x)|^{2} dx
\]
\[
\leq \| f \|^{2} = \int_{I} |\widehat{f}(y)|^{2} dy.
\]
The last inequality can be interpreted as follows. Let
\[
\mu_{l}(y) = \sum_{k=0}^{l}  |\widehat{g}(m^{-k}y)|^{2} \qquad \mbox{if}\quad y \in I.
\]
Then for any $l\in \IN$ the multiplicative operator $T_{l}(\phi)(y) = \mu_{l}(y) \phi(y)$ is a bounded operator $L^{2}(I) \rightarrow L^{2}(I)$ with the norm less than or equal to $1$. Which is true if and only if $\mu_{l}(y) \leq 1$.
\end{proof}

By a simple modification of the last part, related with the application of the Bessel inequality and the proof of Theorem \ref{thm:1} we  obtain the following

\begin{Theorem}\label{thm:2} Let $h^{(\nu)}\in L^2(\mathbb{R}), 1\leq \nu \leq \mu$ and $m= 2,3,\dots$. Suppose that the system $\{h^{(\nu)}_{k,j,m}: k\in \IN_{0}, j\in \IZ,  1\leq \nu \leq \mu\}$  is
 orthonormal. Then
 \[
 \sum_{k=0}^{\infty} \sum_{\nu=1}^{\mu} |\widehat{h^{(\nu)}}(m^{-k}x)|^{2} \leq 1.
 \]
 \end{Theorem}
We formulate the following corollary for $g\in L^1(\mathbb{R})\bigcap L^2(\mathbb{R})$. In the general case a similar result can be proved using the concept of points of approximate continuity (cf. \cite{CKS:01}, \cite{CKS:02}).

\begin{Corollary}\label{cor:1}
Let $g\in L^1(\mathbb{R})\bigcap L^2(\mathbb{R})$ and $m= 2,3,\dots$. Suppose that the system $\{g_{k,j,m}: k\in \IN_{0}, j\in \IZ \}$  is
 orthonormal. Then the continuous function $\widehat{g}$ vanishes at the origin, $\widehat{g}(0) = 0.$
 \end{Corollary}
\begin{proof}
Let $\widehat{g}(0) \neq 0.$ Without loss in generality we can suppose that $\widehat{g}(0) > 0.$ Which yields that $\widehat{g}(y)$ is greater than $\widehat{g}(0)/2$  in a neighborhood of the origin. The last condition
 contradicts to (\ref{eq:1.1}).
\end{proof}
By Corollary \ref{cor:1} we have that when for $g\in L^1(\mathbb{R})\bigcap L^2(\mathbb{R})$ and $m= 2,3,\dots$  the system $\{g_{k,j,m}: k\in \IN_{0}, j\in \IZ \}$  is orthonormal then the set of nontrivial functions $h$ defined on $\IR$ such that
\[
{(a)_{m}} \qquad \qquad \int_{\IR} g_{k,j,m}(x) h(x) dx = 0 \quad \mbox{for all}\quad k,j\in \IZ.
\]
is not empty. Hence, having in mind that the constant function belongs to $L^{\infty}(\IR)$ we obtain
\begin{Corollary}\label{cor:2}
Let $\{ g^{(\nu)}\}_{\nu =1}^{\mu} \subset L^1(\mathbb{R})\bigcap L^2(\mathbb{R})$ and $m= 2,3,\dots$. If the system $\{g^{(\nu)}_{k,j,m}: k\in \IZ, j\in \IZ, 1\leq \nu \leq \mu \}$  is
 orthonormal then it cannot be complete in $L^{1}(\IR)$.
 \end{Corollary}

\

\section{Preliminary results}\label{sec:3}

\subsection{On ${\mathcal M}$-sets}\label{ss:pr}

Further in this section we will consider that $m\geq 2$ is a fixed natural number.
Let ${\mathcal M} = {\mathcal M}(m): = \{ [\frac{j-1}{m^{k}}, \frac{j}{m^{k}}]: k\in \IZ, j\in \IZ \}$. Further, the parameter $m$ will be  omitted to make the notation understandable.
 We will assume that any $m-$adic rational point $\xi = \frac{j}{m^{k}}, k\in \IZ, j\in \IZ $ is ``split" into two distinct points $\xi_{l}$ and $\xi_{r}$ characterized by the following conditions: for any $-\infty<a< \xi <b <+\infty$ we have
\[
\xi_{l} \in (a, \xi],\quad \xi_{l} \notin [\xi, b),\quad \mbox{and}\quad \xi_{r} \in [\xi, b),\quad \xi_{r} \notin (a, \xi].
\]
Hence, there can be easily established one to one correspondence between any $y\in \IR$ and the sequences $\{\Delta_{j}(y)\}_{j=-\infty}^{\infty}\subset {\mathcal M}$ such that
\[
\Delta_{j+1}(y) \subset \Delta_{j}(y) \quad \text{for all}\quad j\in \IZ\quad \text{and} \quad |\Delta_{j}(y)| = m^{-j}.
\]
When talking about the neighborhoods of the points $\xi_{l}$ and $\xi_{r}$, we will understand some intervals $(a,\xi)$ and $(\xi,b)$, respectively.
The measure of the set of all $m-$adic rational points is equal to zero, hence, this assumption will not affect the results that we are going to consider.
In the last section we need concept of ${\mathcal M}$-neighborhoods of $+\infty$ and $-\infty$. For $j\in \IN_{0}$ we put
\[
\Delta_{j}(+\infty)= \IR^{+}\setminus [0,m^{j}], \quad \text{and}\quad \Delta_{j}(-\infty)= \IR^{-}\setminus [-m^{j},0],
\]
where $ \IR^{+}= [0,+\infty)$ and $\IR^{-} =  (\infty,0]$.

\subsubsection{Maximal function}

Let
\begin{equation}\label{max:f}
M_{\mathcal M}f(x) = \sup_{x\in \Delta, \Delta\in \mathcal M} \frac{1}{|\Delta|}\int_{\Delta} |f(t)| dt, \qquad f\in L^{1}_{\text{loc}}(\IR).
\end{equation}

Let us consider also a maximal function with respect to a weight function $\omega$  defined by the following equation
\begin{equation}\label{max:om}
M_{\mathcal M,\omega}f(x) = \sup_{x\in \Delta, \Delta\in \mathcal M} \frac{1}{\omega(\Delta)}\int_{\Delta} |f(t)|\omega(t) dt, \qquad f\in L^{1}_{\text{loc}}(\IR,\omega),
\end{equation}
where $\omega(\Delta) = \int_{\Delta} \omega (t) dt$.
\begin{prop}\label{pr:om}
Let $\omega(x) \geq 0$ for $x\in \IR$, $\omega \in L^{1}_{\text{loc}}(\IR)$ and let $f\in L^{1}_{\text{loc}}(\IR,\omega)$. Then for any  $\lambda >0$
\begin{equation}\label{max:1.1}
\omega(\{t\in \IR: M_{\mathcal M,\omega}f(t) > \lambda\}) \leq \frac{1}{\lambda}\int_{\IR} |f(t)|\omega(t) dt.
\end{equation}
\end{prop}
\begin{proof}
If $x\in \Omega_{\lambda}(f):=\{t\in \IR: M_{\mathcal M,\omega}f(t) > \lambda\}$ then for some $\Delta \in {\mathcal M}$ such that $x\in \Delta $
\[
\frac{1}{\omega(\Delta)}\int_{\Delta} |f(t)|\omega(t) dt >\lambda.
\]
Observe that among all intervals which have the above properties there exists $\Delta_{x} \in {\mathcal M}$  with  maximal $\omega-$measure. Thus, having in mind that $\mathcal M$ is numerable we can find a sequence of mutually disjoint intervals $\{\Delta_{\nu}\} \subset \mathcal M$ so that $\Omega_{\lambda} = \bigcup_{\nu=1}^{\infty} \Delta_{\nu}$.
\end{proof}

 We also have
\begin{prop}\label{pr:maxp}
Let $\omega (x) \geq 0$ for $x\in \IR$, $\omega \in L^{1}_{\text{loc}}(\IR)$.
Then for any $f\in L^{p}_{\text{loc}}(\IR,\omega), p>1$
\begin{equation}\label{max:1.2}
\int_{\IR} M_{\mathcal M,\omega}f(t)^{p} \omega(t) dt \leq  \frac{2^{p}p}{p-1}\int_{\IR} |f(t)|^{p}\omega(t) dt.
\end{equation}
\end{prop}
\begin{proof}
Following the proof of the corresponding result for the Lebesgue measure (see \cite{St:70}, p.7) we split $f$ into two parts, $f= f_{1}+ f_{2}$, where  $f_{1} (x) = f(x)$ if $|f(x)|\geq \frac{\lambda}{2}$ and $f_{1} (x) = 0$ otherwise. Then we have $|f(x)| \leq |f_{1}(x)| + \frac{\lambda}{2}$. Hence, $M_{\mathcal M,\omega}f(x) \leq M_{\mathcal M,\omega}f_{1}(x) + \frac{\lambda}{2}$ and $\Omega_{\lambda}(f) \subseteq \Omega_{\frac{\lambda}{2}}(f_{1})$. Thus by Proposition \ref{pr:om} we have that
\begin{equation}\label{eq:3.5}
\omega(\Omega_{\lambda}(f)) \leq \frac{2}{\lambda}\int_{\IR} |f_{1}(t)|\omega(t) dt = \frac{2}{\lambda} \int_{\Omega_{\frac{\lambda}{2}}(f)} |f(t)|\omega(t) dt.
\end{equation}
Afterwards we have to use the following equality for any measurable function $g:\IR\rightarrow \IR$
\[
\int_{\IR} |g(t)|^{p} \omega(t)dt = p\int_{\IR} \int_{[0,|g(t)|\,]} \lambda^{p-1} d\lambda\, \omega(t)dt
\]
\[
= p\int_{0}^{+\infty} \lambda^{p-1}\omega(\{t: |g(t)| >\lambda\}) d\lambda.
\]
Hence, by (\ref{eq:3.5}) we obtain
\[
\int_{\IR} M_{\mathcal M,\omega}f(t)^{p} \omega(t) dt \leq 2p \int_{0}^{+\infty} \lambda^{p-2}\int_{\Omega_{\frac{\lambda}{2}}(f)} |f(t)|\omega(t) dt d\lambda=
\]
\[
2p\int_{\IR} |f(t)|\omega(t) \int_{0}^{2|f(t)|}\lambda^{p-2} d\lambda\,   dt = \frac{2^{p}p}{p-1} \int_{\IR} |f(t)|^{p}\omega(t) dt
\]
\end{proof}

\subsubsection{Calderon-Zygmund decomposition for $m-$adic intervals}

We need a modified version for the Calderon-Zygmund decomposition (see \cite{Tor:86}) for the $m-$adic intervals. Let $f\in L^{1}[0,1]$, $f\geq 0$ and let $\lambda >0$ is such that
\[
\int_{[0,1]} f(t) dt < \lambda.
\]
At the first step we take $m$ intervals $\{I_{k}\}_{k=1}^{m}\subset {\mathcal M}\bigcap [0,1]$ such that $|I_{k}| = \frac{1}{m}, 1\leq k\leq m$ and $\bigcup_{k=1}^{m} I_{k} = [0,1]$. Let $\{I_{k_{l}}\}_{l=1}^{m_{1}}\subset \{I_{k}\}_{k=1}^{m}$ be all those intervals for which $\eta_{I_{k_{l}}} > \lambda, \, 1\leq l\leq m_{1}<m$, where
\[
\frac{1}{|I| }\int_{I} f(t) dt:= \eta_{I}.
\]
Those intervals are renamed $G_{1}, \ldots, G_{m_{1}}$. Clearly,
\begin{equation}\label{eq:cz1}
 \lambda < \frac{1}{|G_{l}| }\int_{G_{l}} f(t) dt \leq m\lambda.
\end{equation}
If $\eta_{I_{k}} \leq \lambda $ for all $1\leq k \leq m$ then we put $m_{1}=0$.
On the next step we repeat the same procedure on any of those intervals that were not renamed.  The collection of all $m-$adic intervals which are separated on the second step are renamed $G_{m_{1}+1}, \ldots, G_{m_{2}}$. For those intervals the condition (\ref{eq:cz1}) holds again. On the $\nu$th step all $m-$adic intervals which are separated are renamed $G_{m_{\nu -1}+1}, \ldots, G_{m_{\nu }}$. If no any interval is separated then we put $m_{\nu +1} = m_{\nu}$. This procedure produces a collection of disjoint $m-$adic intervals $\{G_{l}\}$ for which  the condition (\ref{eq:cz1}) holds and
\begin{equation}\label{eq:cz2}
|\Omega|= \sum_{l} |G_{l}| < \frac{1}{\lambda }\sum_{l} \int_{G_{l}} f(t) dt \leq \frac{1}{\lambda } \int_{[0,1]} f(t) dt,
\end{equation}
where $\Omega = \bigcup_{l} G_{l}$ and for any $I \in {\mathcal M}\bigcap [0,1]$, $I\subset \Omega^{c}:= [0,1]\setminus \Omega$ we have that $\eta_{I} \leq \lambda.$
The collection $\{G_{l}\}$ will be called Calderon-Zygmund $m-$adic decomposition at level $\lambda$. Let
\begin{equation}\label{eq:cz3}
 g(x) = f(x)\chi_{\Omega^{c}} (x) +  \sum_{l} \eta_{G_{l}} \chi_{G_{l}} (x) \quad \text{and}\quad b(x) = f(x) - g(x).
\end{equation}
We skip the details  of the proof of the following
\begin{prop}\label{pr:cz}
Let $f\in L^{1}[0,1]$, and let  $\lambda >0$ be such that
$
\int_{[0,1]} |f(t)| dt < \lambda.
$
Then there exists a family of disjoint sets $\{G_{l}\}_{l\in \Upsilon}\subset {\mathcal M}$ such that
\begin{equation}\label{eq:cz4}
  |f(x)| \leq \lambda \quad \text{a.e. on}\quad  \Omega^{c},   \quad \text{where}\quad \Omega = \bigcup_{l\in \Upsilon} G_{l},
\end{equation}
(\ref{eq:cz2}) is true and for any $l\in \Upsilon$ holds (\ref{eq:cz1}). Moreover, $f(x) = g(x) + b(x),$ where  $g$ is defined by (\ref{eq:cz3}) and the following conditions hold:
\begin{equation}\label{eq:cz5}
  |g(x)| \leq m\lambda \quad \text{a.e. on}\quad  [0,1];
\end{equation}
\begin{equation}\label{eq:cz6}
  \|g(x)\|^{p}_{p} \leq (m\lambda)^{p-1} \| f\|_{1} \quad \text{for all}\quad  1\leq p <\infty;
\end{equation}
\begin{equation}\label{eq:cz7}
  \int_{G_{l}} b(t) dt = 0  \quad \text{for all }\quad  l\in \Upsilon.
\end{equation}
\end{prop}
 For dyadic intervals a similar result was obtained by C. Watari \cite{Wa:64}.

\subsubsection{Classes of ${\mathcal M}_{p}, p\geq 1$ weights}

\begin{Definition}
We say that a non negative locally integrable function $\omega$ satisfies the condition ${\mathcal M}_{p}$, $p\geq 1$ if
\begin{equation}\label{M:p}
\omega(\Delta) \bigg[\int_{\Delta} \omega^{-\frac{1}{p-1}}(t) dt\bigg]^{p-1} \leq C_{p} |\Delta|^{p} \qquad \forall \Delta \in {\mathcal M},
\end{equation}
where   $C_{p}>0$ is independent of $\Delta\in {\mathcal M}$. For $p=1$ it is understood that
$\bigg[\int_{\Delta} \omega^{-\frac{1}{p-1}}(t) dt\bigg]^{p-1}:= \|\omega^{-1} \|_{L^{\infty}(\Delta)}$.
\end{Definition}
We say that $\omega$ satisfies the condition ${\mathcal M}_{p}(G)$, where $G\subset \IR$ if (\ref{M:p}) holds for all $\Delta \in {\mathcal M}\bigcap G$. The reader should observe that the conditions ${\mathcal M}_{p}([0,1])$ and ${\mathcal M}_{p}((0,1])$ are distinct. In the second case the intervals $[0,2^{-j}], j\in \IN$ should be excluded when one checks the inequality (\ref{M:p}).

The following lemma is obvious.
\begin{Lemma}\label{ap:mpprima}
Let $\omega$ satisfies the condition $\mathcal M_{p}, p>1$ then $\psi = \omega^{-\frac{1}{p-1}}$  satisfies the condition $\mathcal M_{p'}$, where $\frac{1}{p} + \frac{1}{p'} = 1$.
\end{Lemma}
We follow the ideas given in \cite{CF:01} to prove the following result.
\begin{prop}\label{pr:3}
Let $\omega (x) \geq 0$ for $x\in \IR$, $\omega \in L^{1}_{\text{loc}}(\IR)$.
Then for any $f\in L^{p}_{\text{loc}}(\IR,\omega), p>1$
\begin{equation}\label{max:weightp}
\int_{\IR} M_{\mathcal M}f(t)^{p} \omega(t) dt \leq  B_{p}\int_{\IR} |f(t)|^{p}\omega(t) dt,
\end{equation}
for some $B_{p}>0$ independent of $f$ if and only if $\omega$ satisfies the condition ${\mathcal M}_{p}$.
\end{prop}
\begin{proof} Suppose that (\ref{max:weightp}) is true. For any $\Delta \in {\mathcal M}$ we have by (\ref{max:f}) that
\[
      \frac{1}{|\Delta|}\int_{\Delta} |f(t)| dt \chi_{\Delta}(x) \leq    M_{\mathcal M}f(x).
\]
Hence, by (\ref{max:weightp}) we have that
\[
 \bigg( \frac{1}{|\Delta|}\int_{\Delta} |f(t)| dt\bigg)^{p} \omega(\Delta) \leq B_{p} \int_{\IR} |f(t)|^{p}\omega(t) dt.
\]
Letting $f(t) = \omega^{-\frac{1}{p-1}}(t) \chi_{\Delta}(t)$ we obtain (\ref{M:p}) with $C_{p} = B_{p}.$
To prove the opposite assertion one has to use Proposition \ref{pr:om} and the following
\begin{Lemma}\label{ap:eps}
Let $\omega$ satisfies the condition $\mathcal M_{p}, p>1$ then there exists $\varepsilon >0$ such that
$\omega$ satisfies the condition $\mathcal M_{p-\varepsilon}$.
\end{Lemma}
We skip the rest of the proof because the proof in \cite{CF:01} works with small changes.
\end{proof}

\begin{Definition}\label{def:ainf}
We say that a non negative locally integrable function $\omega$ satisfies the condition ${\mathcal M}_{\infty}$ if
there exists $C>0$ and $\delta >0$ such that for any $\Delta \in {\mathcal M} $ and any measurable subset $E\subseteq \Delta$
\begin{equation}\label{M:infty}
\frac{\omega(E)}{\omega(\Delta)}  \leq C \bigg(\frac{|E|}{|\Delta|}\bigg)^{\delta}.
\end{equation}
\end{Definition}
We skip the detailed proofs of the following two lemmas because the corresponding proofs in \cite{CF:01} for $A_{p}$ weights work with obvious changes.
\begin{Lemma}\label{lem:RH}
Let $\omega$ satisfies the condition $\mathcal M_{p}, p>1$ then there exist $r >0$ and $C>0$ such that
\begin{equation}\label{eq:RH}
\bigg(\frac{1}{|\Delta|}\int_{\Delta} \omega(x)^{1+r}\bigg)^{\frac{1}{1+r}}  \leq C \,\frac{\omega(\Delta)}{|\Delta|} \qquad \forall \Delta \in {\mathcal M}.
\end{equation}
\end{Lemma}
\begin{Lemma}\label{lem:apinfty}
Let $\omega$ satisfies the condition $\mathcal M_{p}$ for some  $p>1$ then $\omega$ satisfies the condition $\mathcal M_{\infty}$.
\end{Lemma}
We need also the following
\begin{Lemma}\label{lem:fff}
Let $\omega$ satisfies the condition $\mathcal M_{p}$ for some  $p>1$ then
\[
\omega \notin L(\Delta_{j}(+\infty)), \qquad \omega \notin L(\Delta_{j}(-\infty)), \qquad \forall j\in \IN_{0}.
\]
\end{Lemma}
\begin{proof}
We prove the assertion for the neighborhoods of $+\infty$. For this purpose we observe that there exists $C_{p}>0$ such that for any  $j\in \IN_{0}$
\begin{equation}\label{eq:11}
\omega(\Delta_{-j-1}) \leq C_{p} \omega(\Delta_{-j-1}\setminus \Delta_{-j} ).
\end{equation}
For any  $j\in \IN_{0}$ and any locally integrable function $f\geq 0$ we have that $ M_{\mathcal M}f(x) \geq m^{-j} \int_{[0,m^{j}]} f(t) dt$ if $x\in [0,m^{j}]$. By Proposition \ref{pr:3} we obtain that
\[
\bigg(m^{-j} \int_{[0,m^{j}]} f(t) dt\bigg)^{p} \omega([0,m^{j}]) \leq B_{p} \int_{[0,m^{j}]} f(t)^{p} \omega(t) dt.
\]
Putting $j+1$ instead of $j$ in the above inequality and letting $f$ be the characteristic function of the set $\Delta_{-j-1}\setminus \Delta_{-j}$
we obtain the inequality (\ref{eq:11}). If $\omega \in L(\Delta_{j_{0}}(+\infty))$ then for any $\varepsilon >0$ there exists $N\in \IN$ such that
\[
\int_{[m^{N},+\infty)} \omega(t) dt < \varepsilon.
\]
Which leads to a contradiction with the condition (\ref{eq:11}).
Evidently the proof for the neighborhoods of $-\infty$ is similar.
\end{proof}

\begin{Definition}\label{def:mpc0}
Let $\omega \geq 0$ be a weight function defined on $\IR^{+}$.  We will say that $\omega$ satisfies the condition ${\mathcal M}^{y}_{p}(\IR^{+})$, $p\geq 1$ for some $y\in \IR^{+}$ if
\begin{equation}\label{MyR:p}
\omega(\Delta_{j}(y)) \bigg[\int_{\IR^{+}\setminus \Delta_{j}(y)} \omega^{-\frac{1}{p-1}}(t) dt\bigg]^{p-1} \leq C_{p} |\Delta_{j}(y)|^{p} \qquad \forall j \in \IZ,
\end{equation}
where  $C_{p}>0$ is independent of $j\in \IZ$.
\end{Definition}
We will not formulate the definition of the condition ${\mathcal M}^{y}_{p}(\IR^{-})$ because it is clear from the context.
\begin{Definition}\label{def:mpc}
Let $\omega\geq 0$ be a weight function defined on $\Delta$, where $\Delta \in {\mathcal M},$ $|\Delta| = m^{l} $. We will say that $\omega$ satisfies the condition ${\mathcal M}^{y}_{p}(\Delta)$, $p\geq 1$ for some $y\in \Delta$  if
\begin{equation}\label{My:p}
\omega(\Delta_{j}(y)) \bigg[\int_{\Delta\setminus \Delta_{j}(y)} \omega^{-\frac{1}{p-1}}(t) dt\bigg]^{p-1} \leq C_{p} |\Delta_{j}(y)|^{p} \qquad \forall j >l,
\end{equation}
where  $C_{p}>0$ is independent of $j$.
\end{Definition}

\begin{Lemma}\label{lem:wsing1}
Let $w\geq 0$ be a weight function defined on $[0,1]$ such that $w$ satisfies the condition ${\mathcal M}^{y}_{p}([0,1])$ for some  $y\in [0,1]$ and $ 1 < p < \infty$.
Then there exists $q_{p} >1$ such that
\[
\int_{[0,1]\setminus \Delta_{j+1}(y)} w^{-\frac{1}{p-1}}(t) dt \bigg(\int_{[0,1]\setminus \Delta_{j}(y)} w^{-\frac{1}{p-1}}(t) dt\bigg)^{-1} \geq q_{p}
\]
for all $j\in \IN.$
\end{Lemma}
\begin{proof} For any $j\in \IN$ we have that
\[
\int_{[0,1]\setminus \Delta_{j+1}(y)} w^{-\frac{1}{p-1}}(t) dt \bigg(\int_{[0,1]\setminus \Delta_{j}(y)} w^{-\frac{1}{p-1}}(t) dt\bigg)^{-1}
\]
\[
\geq  1 + \int_{\Delta_{j}(y)\setminus \Delta_{j+1}(y)} w^{-\frac{1}{p-1}}(t) dt \bigg(\int_{[0,1]\setminus \Delta_{j}(y)} w^{-\frac{1}{p-1}}(t) dt\bigg)^{-1}
\]
\[
\geq  1 + \int_{\Delta_{j}(y)\setminus \Delta_{j+1}(y)} w^{-\frac{1}{p-1}}(t) dt \bigg(w(\Delta_{j}(y))\bigg)^{\frac{1}{p-1}} C^{-\frac{1}{p-1}}_{p} |\Delta_{j}(y)|^{-\frac{p}{p-1}}
\]
\[
\geq  1 + |\Delta_{j}(y)\setminus \Delta_{j+1}(y)|^{\frac{p}{p-1}}  C^{-\frac{1}{p-1}}_{p} |\Delta_{j}(y)|^{-\frac{p}{p-1}} \geq  1 + C^{-\frac{1}{p-1}}_{p}\bigg(\frac{m-1}{m}\bigg)^{\frac{p}{p-1}}.
\]
Putting $q_{p} = 1 + C^{-\frac{1}{p-1}}_{p}(\frac{m-1}{m})^{\frac{p}{p-1}}$ we finish the proof.
\end{proof}
Following three lemmas will be used in the last section. In the proofs we will use the following notation: $aE = \{ at: t\in E\}$.
\begin{Lemma}\label{lem:transf1}
Let $\omega \geq 0$ satisfies the condition $\mathcal M_{p}(\IR^{+}), p>1$ with a constant $C_{p}>0$. Then for any $N\in \IN$ the weight function $\omega_{N}(x): = \omega(m^{N}x)$
satisfies the condition $\mathcal M_{p}([0,1])$ with the same constant $C_{p}$.
\end{Lemma}
\begin{proof}
For any $E\in {\mathcal M}\bigcap [0,1]$ we have
\begin{align*}
\omega_{N}(E) \bigg[\int_{E} \omega_{N}^{-\frac{1}{p-1}}(t) dt\bigg]^{p-1} &= m^{-N}\int_{m^{N}E} \omega(x) dx \bigg[m^{-N}\int_{m^{N}E} \omega^{-\frac{1}{p-1}}(x) dx\bigg]^{p-1}\\
&\leq m^{-pN} C_{p}\, |m^{N}E|^{p} = C_{p} |E|^{p}.
\end{align*}
We have used that $m^{N}E \in {\mathcal M}$.
\end{proof}

\begin{Lemma}\label{lem:transf2}
Let $N\in \IN$ and let $y\in [0,m^{N}]$. Suppose that $\omega \geq 0$ satisfies the condition $\mathcal M_{p}([0,m^{N}]\setminus \{y\}), p>1$ with a constant $C_{p}>0$. Then  the weight function $\omega_{N}(x): = \omega(m^{N}x)$
satisfies the condition $\mathcal M_{p}([0,1]\setminus \{y_{N}\})$ with the same constant $C_{p}$, where $y_{N} = m^{-N} y$.
\end{Lemma}
\begin{proof}
For any $E\in {\mathcal M}\bigcap\bigg( [0,1]\setminus \{y_{N}\}\bigg)$ we observe that the interval $m^{N}E \in {\mathcal M}\bigcap\bigg([0, m^{N}]\setminus \{y\}\bigg)$. The rest of the proof is the same as above.
\end{proof}
\begin{Lemma}\label{lem:transf3}
Let $N\in \IN$ and let $y\in [0,m^{N}]$. Suppose that $\omega \geq 0$ satisfies the condition ${\mathcal M}^{y}_{p}([0,m^{N}]), p>1$ with a constant $C_{p}>0$. Then  the weight function $\omega_{N}(x): = \omega(m^{N}x)$
satisfies the condition ${\mathcal M}^{y_{N}}_{p}([0,1])$ with the same constant $C_{p}$, where $y_{N} = m^{-N} y$.
\end{Lemma}
\begin{proof}
For any $j\in \IN$ we have
\begin{align*}
&\omega_{N}(\Delta_{j}(y_{N})) \bigg[\int_{[0,1]\setminus \Delta_{j}(y_{N}) } \omega_{N}^{-\frac{1}{p-1}}(t) dt\bigg]^{p-1} = m^{-N}\int_{\Delta_{j-N}(y)} \omega(x) dx \\
&\times \bigg[m^{-N}\int_{[0,m^{N}]\setminus \Delta_{j-N}(y)} \omega^{-\frac{1}{p-1}}(x) dx\bigg]^{p-1}\\
&\leq m^{-pN} C_{p}\, |\Delta_{j-N}(y)|^{p} = C_{p} |\Delta_{j}(y)|^{p}.
\end{align*}

\end{proof}

\

\subsection{Higher rank Haar wavelets}\label{sec:HW}

We  bring the definition of higher rank Haar wavelets  without recalling the general theory of multiresolution analysis.  For relations of these type of wavelets with $p-$adic analysis see \cite{Koz:01}.
Let $\varphi(x) = \chi_{[0,1]}(x)$ and let
\begin{equation}\label{eq:vm}
V(m) = \mbox{span} \{\varphi_{1,j,m}(x): 0\leq j\leq m-1\}
\end{equation}
for any $m=2,3,\ldots$. Afterwards, let $\{h^{(\nu)}(x): 0\leq \nu \leq m-1\}$ be an orthonormal basis in $V(m)$ such that $h^{(0)}(x) = \varphi(x)$. The system
\begin{equation}\label{haar:m}
H(m) = \{h^{(\nu)}_{k,j,m}(x): k\in \IZ; j\in \IZ; 1\leq \nu \leq m-1 \},
\end{equation}
where
\begin{equation}\label{haar:m1}
h^{(\nu)}_{k,j,m}(x)= m^{k/2}h^{(\nu)}(m^{k}x-j)
\end{equation}
will be called $m$th rank Haar system. Sometimes we will use also the following notation
\begin{equation}\label{haar:m2}
h^{(\nu)}_{\Delta }:= h^{(\nu)}_{k,j,m}(x)\qquad \text{when}\quad \Delta = [\frac{j}{m^{k}}, \frac{j+1}{m^{k}}].
\end{equation}
 The orthogonality of the system (\ref{haar:m}) is obvious.

\begin{Theorem}\label{thm:com}
 The system $H(m)$ is complete in $L^{p}(\IR), 1<p<\infty.$
\end{Theorem}
\begin{proof}
Let $p, 1<p<\infty$ be fixed.
It is easy to observe that the proof will be finished if we show that for any $\varphi_{1,j,m}(x), 0\leq j\leq m-1$ and any $\varepsilon >0$ there exists a finite linear combination $P_{j}$ of functions\newline
 $\{h^{(\nu)}_{k,l,m}(x): k\in \IZ\setminus\IN_{0}; l\in \IZ; 1\leq \nu \leq m-1 \}$ such that $\|\varphi_{1,j,m} - P_{j} \|_{L^{p}(\IR)} < \varepsilon$.

 Let $l\in\IN$ be such that $m^{1/2} m^{l(1/p-1)} < \varepsilon$.
We set
\[
V^{(l)}(m) = \mbox{span} \{\varphi_{1,j,m}(x): 0\leq j\leq m^{l+1}-1\}.
\]
It is clear that $\dim V^{(l)}(m) = m^{l+1}$. Let us show by induction that there are exactly  $m^{l+1}-1$ functions from the system\newline $\{h^{(\nu)}_{k,j,m}(x): 1\geq  k\geq -l+1; j\in \IZ; 1\leq \nu \leq m-1 \}$ with supports in $[0,m^{l}]$.

If $l=0$ then it is obvious. Suppose that for some $\mu\in \IN$ we have that the number  of functions from the system
\begin{equation}\label{sis:1}
\{h^{(\nu)}_{k,j,m}(x): 1\geq  k\geq -\mu+1; j\in \IZ; 1\leq \nu \leq m-1 \}
\end{equation}
 with  supports in $[0,m^{\mu}]$ is equal to $m^{\mu+1}-1$. Then it is clear that there are $(m^{\mu+1}-1)m$ functions from the system (\ref{sis:1})
  that have their supports in $[0,m^{\mu +1}]$. Note that the functions $\{h^{(\nu)}_{-\mu,0,m}(x):  1\leq \nu \leq m-1 \}$ vanish outside the closed interval $[0,m^{\mu +1}]$. Thus we have $(m^{\mu+1}-1)m +m-1 = m^{\mu+2}-1$ mutually orthogonal functions in (\ref{sis:1}) which have their supports in $[0,m^{\mu +1}]$.

Let $\{ g_{i} \}_{i=1}^{m^{l+1}-1}$ be all functions from the system $$\{h^{(\nu)}_{k,j,m}(x): 1\geq  k\geq -l+1; j\in \IZ; 1\leq \nu \leq m-1 \}$$ that have their supports in $[0,m^{l}]$. Evidently $\{ g_{i} \}_{i=1}^{m^{l+1}-1} \subset V^{(l)}(m).$ \newline
Let $g_{0}(x) = \chi_{[0,m^{l}]}(x)$. Then $g_{0} \in V^{(l)}(m)$ and
$g_{0}$ is orthogonal to all elements of $\{ g_{i} \}_{i=1}^{m^{l+1}-1}$. Hence, $\{ g_{i} \}_{i=0}^{m^{l+1}-1}$ is a basis in $V^{(l)}(m)$ and
\[
 \varphi_{1,j,m} = \sum_{i=0}^{m^{l+1}-1} a^{(j)}_{i} g_{i},\qquad \text{where}\quad a^{(j)}_{0} = m^{-l}\int_{[0,m^{l}]} \varphi_{1,j,m}(t) dt.
\]
Thus we obtain that for any $0\leq j\leq m-1$
\[
\bigg\| \varphi_{1,j,m} - \sum_{i=1}^{m^{l+1}-1} a^{(j)}_{i} g_{i} \bigg\|_{L^{p}(\IR)} = \bigg\|  a^{(j)}_{0} g_{0}\bigg\|_{L^{p}(\IR)}  = m^{1/2} m^{l(1/p-1)} < \varepsilon.
\]

\end{proof}

\

\section{$m$th rank Haar system on $[0,1]$}\label{sec:mH}

Let ${\textsf{h} }_{0}(x) \equiv 1$ for $x\in [0,1]$. For any $n\in \IN$ we have a unique representation
\begin{equation}\label{n:1}
n= m_{k} + j-1, \quad \text{where}\quad k\in \IN,\, 1\leq j\leq m^{k},
\end{equation}
and
\begin{equation}\label{n:2}
m_{k} = 1+ m + m^{2}+ \cdots + m^{k-1}, \quad m_{1} =1.
\end{equation}
 For any $1\leq \nu \leq m -1$ we put
\[
{\textsf{h} }^{(\nu)}_{n}(x) = h^{(\nu)}_{k,j-1,m}(x) \quad \text{for}\, x\in [0,1].
\]
Afterwards we enumerate the functions in the following way
\begin{align}\label{haar:1}
{\textsf{h} }_{l}(x) &= {{h} }^{(l)}(x)\quad \text{for}\quad  1\leq l\leq m-1;\\
{\textsf{h} }_{l}(x) &= {\textsf{h} }^{(\nu)}_{n}(x)\quad \text{for}\quad  l= \nu + n(m-1),\, n\in \IN.
\end{align}
We denote the $m$th rank Haar system by ${\mathcal H}(m) = \{ {\textsf{h} }_{l}(x) \}_{l=0}^{\infty}$. We also let
\begin{equation}\label{n:k}
\mu_{0 } = 0, \, \mu_{1 } = \mu_{0 } + m-1, \cdots,  \mu_{k+1} = \mu_{k} + (m -1) m^{k}, \cdots .
\end{equation}

The following lemma is the analogue of  Schauder's lemma  for the classical Haar system (see \cite{Sch:1}).
For any $f\in L^{1}[0,1]$ and for any $1\leq j\leq m^{k}, k\in \IN$  we put
\[
\Theta_{\mu_{k} + j(m-1)}(f,x) =    \sum_{l=0}^{\mu_{k}} a_{l}(f) {\textsf{h} }_{l}(x) +  \sum_{s=0}^{j-1} \sum_{\nu =1}^{m-1} a^{(\nu)}_{k,s,m}(f) h^{(\nu)}_{k,s,m}(x),
\]
where
\[
  a_{l}(f) =\int_{[0,1]} f(t) {\textsf{h} }_{l}(t) dt; \quad  a^{(\nu)}_{k,s,m}(f) =\int_{[0,1]} f(t) h^{(\nu)}_{k,s,m}(t) dt.
\]

\begin{Lemma}\label{Hlem:1}
Let $f\in L^{1}[0,1]$ and let $1\leq j\leq m^{k}, k\in \IN_{0}$. Then the partial sum  $\Theta_{\mu_{k} + j(m-1)}(f,x)$ is constant on any interval from the collection of sets
\begin{equation}\label{delta1}
\bigg \{\bigg[\frac{s}{m^{k+1}}, \frac{s+1}{m^{k+1}}\bigg]: 0\leq s \leq jm-1\bigg\},
 \end{equation}
 \begin{equation}\label{delta2}
\bigg\{\bigg[\frac{l}{m^{k}}, \frac{l+1}{m^{k}}\bigg]: j\leq l \leq m^{k}-1\bigg\}.
 \end{equation}
Moreover, for any $\Delta$ from (\ref{delta1}) or from (\ref{delta2})
\begin{equation}\label{delsu}
\Theta_{\mu_{k}+ j(m-1) }(f,x) =  \frac{1}{|\Delta|}\int_{\Delta} f(t) dt\quad \text{for}\quad x\in \Delta.
\end{equation}
\end{Lemma}
\begin{proof}
At first we show that the assertion of the lemma is true for $\Theta_{\mu_{k}}(f,x),$ $ k\in \IN_{0}$. Indeed,
\[
\text{span}\{ {\textsf{h} }_{l}, 0\leq l\leq m^{k}-1 \} = V(m),
\]
where $V(m)$ is defined by (\ref{eq:vm}). Hence, for any   $  \Delta$  from (\ref{delta2}) with $j=1$ we have that  for $x\in \Delta$
\[
\Theta_{\mu_{k}}(f,x) = \sum_{l=0}^{m^{k}-1} \int_{[0,1]} f(t) \varphi_{k,l,m}(t) dt\, \varphi_{k,l,m}(x) = \frac{1}{|\Delta|}\int_{\Delta} f(t) dt.
\]
Afterwards we observe that in the general case
\[
\Theta_{\mu_{k}+ j(m-1)}(f,x) = \Theta_{\mu_{k+1}}(f,x)\quad \text{if}\quad x\in [0,\frac{j}{m^{k}}]
\]
and
\[
\Theta_{\mu_{k}+ j(m-1)}(f,x) = \Theta_{\mu_{k}}(f,x)\quad \text{if}\quad x\in [\frac{j}{m^{k}}, 1].
\]
Which finishes the proof.
\end{proof}

By Lemma \ref{Hlem:1} we obtain the following corollaries.
\begin{Corollary}\label{Hcor:1}
For any $m=2,3,\ldots$ the system ${\mathcal H}(m)$ is a basis in any space $L^{p}[0,1], 1\leq p < \infty.$
\end{Corollary}
\begin{proof}
By Lemma \ref{Hlem:1} as in the case of the classical Haar system we have that
\[
\|\Theta_{\mu_{k}+ j(m-1)} \|_{L^{p}\to L^{p}} \leq 1\qquad \text{for all}\quad 1\leq j\leq m^{k}, k\in \IN.
\]
To finish the proof we have to check that $\lim_{l\to \infty}|a_{l}(f)| \| {\textsf{h} }_{l}\|_{L^{p}[0,1]} = 0$. We skip the technical details because afterwards we are going to return to the similar question in the weighted norm case.
\end{proof}
\begin{Corollary}\label{Hcor:2}
For any $f\in L^{1}[0,1]$ the Fourier series of $f$ with respect to the system  ${\mathcal H}(m),$ $m=2,3,\ldots$ converges almost everywhere to $f$ on $[0,1]$.
\end{Corollary}
\begin{proof}
For every $x\in [0,1]$ which is a Lebesgue point of $f$ we have that
\[
\lim_{k\to \infty} \Theta_{\mu_{k}}(f,x) = f(x).
\]
Afterwards we observe that $|a_{l}(f)| |{\textsf{h} }_{l}(x)| \leq C M_{\mathcal M}(f,x)$ which finishes the proof.
\end{proof}
For any $k\in \IN_{0}$ and any $1\leq j\leq m^{k}$ consider the kernel
\begin{equation}\label{ker}
K_{kj}(t,x) = \sum_{l=0}^{\mu_{k}} {\textsf{h} }_{l}(t) {\textsf{h} }_{l}(x) +  \sum_{s=0}^{j-1} \sum_{\nu =1}^{m-1} h^{(\nu)}_{k,s,m}(t) h^{(\nu)}_{k,s,m}(x).
\end{equation}
Let $\{ G_{i}: 1\leq i\leq m^{k} + j(m-1)\}$ be mutually disjoint sets from (\ref{delta1}) and (\ref{delta2}).
Further in the paper we will need the following result.
\begin{Lemma}\label{Hlem:ker1}
Let $k\in \IN_{0}$ and $1\leq j\leq m^{k}$. Then  the kernel
\begin{equation}\label{ker:1}
K_{kj}(t,x) =   \frac{1}{|G_{i}|}\quad \text{for}\quad (t,x)\in G_{i}^{2}, 1\leq i\leq m^{k} + j(m-1);
\end{equation}
and
\[
K_{kj}(t,x) = 0, \qquad \text{if}\quad (t,x)\in [0,1]^{2}\setminus \bigcup_{i=1}^{m^{k} + j(m-1)} G^{2}_{i}.
\]
\end{Lemma}
\begin{proof}
We have that  $\{ |G_{i}|^{-1/2}\chi_{G_{i}}(x) \}_{i=1}^{m^{k} + j(m-1)}$ is an orthonormal system of functions. From Lemma \ref{Hlem:1} it follows  that the orthonormal system of functions $\{{\textsf{h} }_{l}(x)\}_{l=0}^{\mu_{k}} \bigcup \{ h^{(\nu)}_{k,s,m}(x): 0\leq s\leq j-1, 1\leq \nu \leq m-1\}$ can be obtained from the set of functions $\{ |G_{i}|^{-1/2}\chi_{G_{i}}(x) \}_{i=1}^{m^{k} + j(m-1)}$ by an orthogonal transformation.  Hence,
\[
K_{kj}(t,x) = \sum_{i=1}^{m^{k} + j(m-1)} |G_{i}|^{-1/2}\chi_{G_{i}}(t) |G_{i}|^{-1/2}\chi_{G_{i}}(x)
\]
\[
= \sum_{i=1}^{m^{k} + j(m-1)} |G_{i}|^{-1}\chi_{G_{i}}(t) \chi_{G_{i}}(x).
\]
\end{proof}

\begin{Definition}\label{def:tot}
We say that a system of functions $\{\phi_{k}\}_{k=1}^{\infty} \subset L^{\infty}[0,1]$ is total with respect to $L^{1}[0,1]$ if
\begin{equation}\label{tot:1}
\int_{[0,1]}  f(t) \phi_{k}(t) dt = 0 \qquad \text{for all}\quad k\in \IN \quad \text{for some} \quad f\in  L^{1}[0,1]
\end{equation}
if and only if $f= 0$ a.e. on $ [0,1]$.
\end{Definition}
By Lemma \ref{Hlem:1} it follows immediately
\begin{Corollary}\label{Htot:1}
The system  ${\mathcal H}(m),$ $m=2,3,\ldots$ is total with respect to $ L^{1}[0,1]$.
\end{Corollary}

\begin{Theorem}\label{UnB:1}
For any $m=2,3,\ldots$ the system ${\mathcal H}(m)$ is an unconditional basis in any space $L^{p}[0,1], 1< p < \infty.$
\end{Theorem}

The reader can find well known facts about unconditional bases in \cite{LZ:77}.
For any sequence $\epsilon = \{\epsilon_{l}\}_{l=0}^{\infty}$, where $\epsilon =\pm 1$ we consider an operator $I_{\epsilon}: L^{1}[0,1] \rightarrow L^{0}[0,1] $ defined as follows $I_{\epsilon}(f,x) = \sum_{l=0}^{\infty} \epsilon_{l}a_{l}(f) {\textsf{h} }_{l}$.
\begin{prop}\label{eps:1}
The operator $I_{\epsilon}$ is of weak-$(1,1)$ type.
\end{prop}
\begin{proof}
We adopt the idea of the proof  given in \cite{Wa:64}. Let $f\in L^{\infty}[0,1]$ and suppose that $\lambda > \|f\|_{1}$.
Without loss in generality we can suppose that  $f\geq 0$ (see \cite{St:70}, pp. 21--22). By Proposition \ref{pr:cz} we write $f(x) = g(x) + b(x)$, where $g$ satisfies the condition (\ref{eq:cz5}). The system ${\mathcal H}(m)$ is a complete orthonormal system. Hence, $I_{\epsilon}: L^{2}[0,1] \rightarrow L^{2}[0,1] $ is an isometry. Thus by the Tchebychev inequality we will have
\begin{equation}\label{eq:eg}
|\{ x\in [0,1]: |I_{\epsilon}(g,x)| > \lambda\}| \leq \frac{1}{\lambda^{2}} \int_{[0,1]} g^{2}(x) dx \leq \frac{m}{\lambda}\|g\|_{1} \leq \frac{m}{\lambda}\|f\|_{1},
\end{equation}
where the last inequality follows by (\ref{eq:cz6}).Afterwards, we apply the following property of $m-$adic intervals. If $\Delta_{1}, \Delta_{2} \subset {\mathcal M}$ then only two relations are possible or $\Delta_{1}\bigcap \Delta_{2}= \emptyset$ or one of those intervals is a subset of another interval. By the definition of the system ${\mathcal H}(m)$ and by (\ref{eq:cz7}) it is easy to deduce that $I_{\epsilon}(b,x) = 0$ for $x\in \Omega^{c}$. Thus by (\ref{eq:cz2}) and  (\ref{eq:eg}) we obtain
\begin{equation}\label{eq:eg1}
|\{ x\in [0,1]: |I_{\epsilon}(f,x)| > \lambda\}| \leq \frac{m+1}{\lambda}\|f\|_{1}.
\end{equation}
From the last inequality readily follows  that for any $f\in L^{1}[0,1]$ the series $\sum_{l=0}^{\infty} \epsilon_{l}a_{l}(f) {\textsf{h} }_{l}$ converges in measure on $[0,1]$. Observe that in the proof of the inequality (\ref{eq:eg1}) the condition $f\in L^{\infty}[0,1]$ was used only to claim the existence of $I_{\epsilon}(f,x)$. Hence, the proof is complete.
\end{proof}
The analogue of Proposition \ref{eps:1} for the Haar system was obtained by S. Yano \cite{Ya:59}.
\begin{proof}
By Proposition \ref{eps:1} and the Marcinkiewicz interpolation theorem (see \cite{Zy:59}) we obtain that ${\mathcal H}(m)$ is an unconditional basis in  $L^{p}[0,1], 1< p \leq 2.$ Afterwards by duality we finish the proof of Theorem \ref{UnB:1}.
\end{proof}
For the system ${\mathcal H}(m)$ we put
\[
G_{m}(f,x) = \bigg(\sum_{l=0}^{\infty} |a_{l}(f) {\textsf{h} }_{l}(x)|^{2}\bigg)^{\frac{1}{2}},\qquad \text{where}\quad f\in L^{1}[0,1].
\]
For the operator $G_{m}: L^{1}[0,1] \rightarrow L^{0}[0,1] $ the following proposition holds.
\begin{prop}\label{Gm:1}
The operator $G_{m}$ is of weak-$(1,1)$ type.
\end{prop}
\begin{proof} Let $f= \sum_{l=0}^{N} a_{l} {\textsf{h} }_{l}$ be any polynomial with respect to the system ${\mathcal H}(m)$ and let $\epsilon = \{\epsilon_{l}\}_{l=0}^{N}$ be any Rademacher sequence. By a well known inequality (see \cite{Kah:93},p.8) we have that for any $0<\alpha<1$ and and any $x\in [0,1]$
\[
\mathcal{P}( I_{\epsilon}(f,x)> \alpha G_{m}(f,x)) >\frac{1}{3}(1-\alpha)^{2}.
\]
Observe that $ G_{m}(I_{\epsilon}(f,\cdot),x)= G_{m}(f,x)$ for any Rademacher sequence $\epsilon.$     For any $\lambda >0$ we have that if the following two events $\{ I_{\epsilon}(f,x)> \alpha G_{m}(f,x)\}$, $\{ G_{m}(f,x)> \frac{\lambda}{\alpha} \}$ then $\{ I_{\epsilon}(f,x)> \lambda\}.$ Hence, by  Proposition \ref{eps:1} we finish the proof for the polynomials with respect to the system ${\mathcal H}(m)$.

For arbitrary $f\in L^{1}[0,1]$ we have that the sequence $G_{m}(\Theta_{n}(f,\cdot),x) $ is an increasing sequence which a.e. converges to $G_{m}(f,x).$ Hence,
\begin{gather*}
|\{ G_{m}(f,x)> \lambda \}| = \lim_{n\to +\infty} |\{ G_{m}(\Theta_{n}(f,\cdot),x) > \lambda \}| \\< \frac{C}{\lambda}\sup_{n} \| \Theta_{n}(f,\cdot)\|_{L^{1}[0,1]} \leq \frac{C_{1}}{\lambda} \| f\|_{L^{1}[0,1]}
\end{gather*}
\end{proof}

By standard arguments (see \cite{Wo:97}) one can derive from Theorem \ref{UnB:1}
 that for all $1<p<\infty$
\begin{equation}\label{eq:sqin}
\|G_{m}(f,\cdot) \|_{L^{p}[0,1]} \simeq \|f \|_{L^{p}[0,1]}.
\end{equation}

\

\subsection{ Haar wavelet systems as unconditional bases in $L^{p}(\IR), 1< p < \infty.$ }

From Theorem \ref{UnB:1} we easily derive

\begin{Theorem}\label{WUnB:1}
For any $m=2,3,\ldots$ the system ${ H}(m)$ is an unconditional basis in any space $L^{p}(\IR), 1< p < \infty.$
\end{Theorem}

Further  we will use the following notations:
$\IR^{+} = [0,+\infty)$, $\IR^{-}= (-\infty,0]$ and $\IZ^{+} = \IR^{+}\bigcap \IZ$  having in mind the agreement introduced in Section \ref{ss:pr}.
For technical reasons we divide the system $H(m)$ into two parts:
\begin{equation}\label{haar1:m}
H^{+}(m) = \{h^{(\nu)}_{k,j,m}(x): k\in \IZ; j\geq 0; 1\leq \nu \leq m-1 \},
\end{equation}
\begin{equation}\label{haar2:m}
H^{-}(m) = \{h^{(\nu)}_{k,j,m}(x): k\in \IZ; j\leq -1; 1\leq \nu \leq m-1 \}.
\end{equation}
We are going to show that the systems $H^{+}(m)$, $H^{-}(m)$ are unconditional bases respectively in the spaces $L^{p}(\IR^{+})$ and $L^{p}(\IR^{-}), 1< p < \infty.$ Let us prove the following
\begin{Theorem}\label{W1UnB:1}
For any $m=2,3,\ldots$ the system $H^{+}(m)$ is an unconditional basis in any space $L^{p}(\IR^{+}), 1< p < \infty.$
\end{Theorem}
\begin{proof}
Let $f\in L^{p}(\IR^{+})$ and
let $\Omega \subset \IZ\times \IZ^{+}$ be a finite set. For any $1\leq \nu \leq m-1$ consider the sum
\[
S^{(\nu)}_{\Omega}(f,x) = \sum_{(k,j)\in \Omega} c^{(\nu)}_{kj}(f) h^{(\nu)}_{k,j,m}(x),
\]
where
\[
c^{(\nu)}_{kj}(f) = \int_{\IR^{+}} f(t) h^{(\nu)}_{k,j,m}(t) dt.
\]
Let $N\in \IN$ be such that for all $(k,j)\in \Omega$ $h^{(\nu)}_{k,j,m}(x) = 0$ if $x\in [m^{N}, +\infty)$. Consider the dilation operator $D_{N}(\phi)(x) = m^{\frac{N}{2}}\phi(m^{N}x)$. It is clear that $D_{N}(h^{(\nu)}_{k,j,m}) \in {\mathcal H}(m)$ for any $h^{(\nu)}_{k,j,m}$ which satisfies to the above conditions if we consider the restriction on $[0,1]$ of the image of the operator.  Thus
\[
D_{N}(S^{(\nu)}_{\Omega}(f,\cdot))(x) = \sum_{(k,j)\in \Omega} c^{(\nu)}_{kj}(f) D_{N}(h^{(\nu)}_{k,j,m})(x)
\]
on $[0,1]$ is a finite linear combination of elements from  ${\mathcal H}(m)$. We also have that if $(k,j)\in \Omega$
\[
c^{(\nu)}_{kj}(f) = \int_{[0,m^{N}]} f(t) h^{(\nu)}_{k,j,m}(t) dt = \int_{[0,1]} D_{N}(f)(t)  D_{N}(h^{(\nu)}_{k,j,m})(t) dt.
\]
Hence, $D_{N}(S^{(\nu)}_{\Omega}(f,\cdot))(x)$ on $[0,1]$ coincides with the sum of a subsequence of the expansion of the function $D_{N}(f)$ with respect to the system ${\mathcal H}(m)$. By Theorem \ref{WUnB:1} we obtain that there exists $C_{p}>0$ which depends only on $p$ such that
\[
\int_{[0,1]} |D_{N}(S^{(\nu)}_{\Omega}(f,\cdot))(t)|^{p} dt \leq C^{p}_{p} \int_{[0,1]} |D_{N}(f)(t)|^{p} dt
\]
which yields
\[
\| S^{(\nu)}_{\Omega}(f,\cdot)\|_{L^{p}(\IR^{+})} \leq C_{p} \| f\|_{L^{p}(\IR^{+})}.
\]

\end{proof}
It is clear that in a similar way we can check that $H^{-}(m)$ is an unconditional basis in any space $L^{p}(\IR^{-}), 1< p < \infty.$
Thus we the proof of Theorem \ref{WUnB:1} is finished.

\

\subsection{$m$th rank Haar system in $L^{p}([0,1],w), 1\leq p < \infty.$ }

\begin{Theorem}\label{WHB:1}
For any $m=2,3,\ldots$ the system ${\mathcal H}(m)$ is a basis in the weighted norm space $L^{p}([0,1],w),$ $ 1\leq p < \infty$ if and only if $w$ satisfies the condition ${\mathcal M}_{p}([0,1])$.
\end{Theorem}
\begin{proof}
By Corollary \ref{Htot:1} we easily obtain that the system ${\mathcal H}(m)$ is complete in $L^{p}([0,1],w),$ $ 1\leq p < \infty$.
Suppose that $w$ satisfies the condition $\mathcal M_{p}([0,1])$. Then it is evident that
\[
\frac{\chi_{[0,1]}}{w} \in L^{p'}([0,1],w), \quad \text{where}\quad \frac{1}{p} +\frac{1}{p'} = 1.
\]
Hence, the system ${\mathcal H}(m) = \{ {\textsf{h} }_{l}(x) \}_{l=0}^{\infty}$ is minimal in  $L^{p}([0,1],w)$   and its conjugate system is the system ${\mathcal H}^{*}(m) = \{ \frac{1}{w(x)}{\textsf{h} }_{l}(x) \}_{l=0}^{\infty}$. Thus for any $f\in L^{p}([0,1],w),$ the coefficients of its expansion with respect to the system ${\mathcal H}(m)$ are equal to
\[
b_{l}(f) = \int_{[0,1]} f(t) \frac{1}{w(t)}{\textsf{h} }_{l}(t) w(t) dt = \int_{[0,1]} f(t) {\textsf{h} }_{l}(t) dt = a_{l}(f).
\]
Hence, for any $ k\in \IN$ and $1\leq j \leq m^{k}$  the partial sums of the mentioned expansion  with indices $\mu_{k}+ j(m-1)$ coincide with $\Theta_{\mu_{k}+ j(m-1)}(f,x)$ (see subsection \ref{sec:mH}). By Lemma \ref{Hlem:1} it follows easily that
\begin{equation}\label{eq:bound}
\|\Theta_{\mu_{k}+ j(m-1)}(f,\cdot)\|_{L^{p}([0,1],w)} \leq C \|f\|_{L^{p}([0,1],w)},
\end{equation}
where $C=C(w,p,m)$ is independent of $f$. If we prove that
\begin{equation}\label{eq:gent}
\lim_{l\to +\infty} |b_{l}(f)| \|{\textsf{h} }_{l}\|_{L^{p}([0,1],w)} = 0
\end{equation}
then it will follow that (\ref{eq:bound}) holds for all $n\in \IN$. Which yields that  the system ${\mathcal H}(m)$ is a basis  $L^{p}([0,1],w).$ We have that
$b_{l}(f) = \int_{[0,1]} [f(t)-P(t)] {\textsf{h} }_{l}(t) dt $ for any $P(t) = \sum_{k=0}^{l-1} d_{k} {\textsf{h} }_{k}(t)$. If ${\textsf{h} }_{l}(x) = h^{(\nu)}_{k,j,m}(x) \quad \text{for}\, x\in [0,1]$ and $\Delta = [\frac{j}{m^{k}}, \frac{j+1}{m^{k}}]$ then
\begin{gather*}
|b_{l}(f)| \|{\textsf{h} }_{l}\|_{L^{p}([0,1],w)} \leq \\
\leq \| f-P\|_{L^{p}([0,1],w)} \bigg\|\frac{h^{(\nu)}_{k,j,m}}{w^{\frac{1}{p}}}\bigg\|_{L^{p'}(\Delta)} \|h^{(\nu)}\|_{L^{\infty}([0,1])} m^{\frac{k}{2}} [w(\Delta)]^{\frac{1}{p}}\\
\leq \| f-P\|_{L^{p}([0,1],w)} \|h^{(\nu)}\|^{2}_{L^{\infty}([0,1])} m^{k} [w(\Delta)]^{\frac{1}{p}} \bigg[\int_{\Delta} {w^{-\frac{1}{p-1}}}\bigg]^{\frac{1}{p'}} \\
\leq \| f-P\|_{L^{p}([0,1],w)} \|h^{(\nu)}\|^{2}_{L^{\infty}([0,1])} C^{\frac{1}{p}}_{p}.
\end{gather*}
The last inequality yields (\ref{eq:gent}) because the system ${\mathcal H}(m)$ is complete in $L^{p}([0,1],w).$

To prove the necessity suppose that the system ${\mathcal H}(m)$ is a basis in the weighted norm space $L^{p}([0,1],w),$ where $ 1\leq p < \infty$.\newline Let ${\mathcal H}^{*}(m)= \{ {\textsf{h} }^{*}_{l}(x) \}_{l=0}^{\infty}$ be the conjugate system of the basis ${\mathcal H}(m)$. Then we have that
\[
\int_{[0,1]}[{\textsf{h}}^{*}_{0}(t)w(t) - 1] {\textsf{h} }_{l}(t)  dt = 0 \qquad \text{for all}\quad l\in \IN_{0}.
\]
Hence, ${\textsf{h}}^{*}_{0}(t) = \frac{1}{w(t)} \in L^{p'}([0,1],w).$ Thus we obtain that
\[
{\textsf{h}}^{*}_{l}(x) = \frac{{\textsf{h} }_{l}(x)}{w(x)} \qquad \text{for all}\quad l\in \IN_{0}.
\]
Thus for any $f\in L^{p}([0,1],w)$ $n$th partial sums of its expansion with respect to the basis ${\mathcal H}(m)$ coincide with $\Theta_{n}(f,x)$. By Lemma \ref{Hlem:1} it follows that for some $C\geq 1$ such that for any $\Delta \in {\mathcal M}$
\[
\sup \frac{1}{|\Delta|^{p}} \bigg|\int_{\Delta} f(t) dt\bigg|^{p} w(\Delta) \leq C^{p},
\]
where the supremum is taken over all  $\| f\|_{L^{p}([0,1],w)} \leq 1$. The last inequality easily yields (\ref{M:p}) with $C_{p}= C^{p}$.
\end{proof}

The prove of the following result technically is much more complicated. The main line of our proof is close to the one given in \cite{GW:74}(see also \cite{CF:01} and \cite{K:2},\cite{GK:94}).
\begin{Theorem}\label{WHUB:1}
For any $m=2,3,\ldots$ the system ${\mathcal H}(m)$ is an unconditional basis in the weighted norm space $L^{p}([0,1],w),$ $ 1< p < \infty$ if and only if $w$ satisfies the condition ${\mathcal M}_{p}([0,1])$.
\end{Theorem}

\begin{Lemma}\label{lem:HGlambda}
Suppose that $w$ is a weight function  which satisfies the condition ${\mathcal M}_{\infty}([0,1]).$  Then for any $\lambda>0$, any  $0< \gamma <1$ and
for any $f\in L^{1}[0,1]$
\begin{align*}\label{eq:glam}
w(\{x\in [0,1]: G_{m}(f,x)>& 2\lambda\quad \text{and}\quad M_{\mathcal M}f(x) \leq \gamma \lambda\}) \\ &\leq C\gamma^{\delta} w(\{x\in [0,1]: G_{m}(f,x)> \lambda\}),
 \end{align*}
where $C>0$ is independent of $f$, $\lambda >0$ and  $\gamma >0,$ while $\delta >0$ is the corresponding constant from Definition \ref{def:ainf}.
\end{Lemma}
In the formulation of the following assertion we use the agreement formulated in the Section \ref{ss:pr} .
\begin{Lemma}\label{lem:Mint}
For any $f\in L^{1}[0,1]$ and any $\lambda>0$ there exists a finite or denumerable set of disjoint closed intervals $\{ \Delta_{k} \}_{k\in \Upsilon}$ such that the set
\begin{equation}\label{eq:El}
E_{\lambda}(f) = \{x\in [0,1]: G_{m}(f,x)> \lambda\} = \bigcup_{k\in \Upsilon} \Delta_{k}.
\end{equation}
\end{Lemma}
\begin{proof}
According to our agreement for any $x\in [0,1]$ there exists a unique sequence of closed intervals $\Gamma_{k}(x) \subset {\mathcal M}$ such that $\Gamma_{k}(x) \subset \Gamma_{k-1}(x)$ for all $k\in \IN$ and $|\Gamma_{k}(x)| =m^{-k}, \bigcap_{k=0}^{\infty} \Gamma_{k}(x) = x. $ For any $x_{0}\in E_{\lambda}(f)$ there exists $k(x_{0})\in \IN_{0}$ so that $\Gamma_{k(x_{0})}(x_{0}) \subseteq E_{\lambda}(f)$ and $\Gamma_{k(x_{0})-1}(x_{0})$ at least contains a point which does not belong to $ E_{\lambda}(f)$. Indeed, if $G_{m}(f,x_{0})> \lambda$ then there exists $N\in \IN$ such that
\[
[G_{m}(f,x_{0})]^{2} \geq \sum_{l=0}^{N} |a_{l}(f) {\textsf{h} }_{l}(x_{0})|^{2}  > \lambda^{2}.
\]
Hence, for some $\nu \in \IN$ the sum  $\sum_{l=0}^{N} |a_{l}(f) {\textsf{h} }_{l}(x_{0})|^{2}$ is constant on $\Gamma_{\nu}(x_{0})$. Which means that $\Gamma_{\nu}(x_{0}) \subseteq E_{\lambda}(f). $ The number $k(x_{0})$ will be the smallest index for which the last relation holds. Afterwards, one observes that $\max_{x\in E_{\lambda}(f)} |\Gamma_{k(x)}(x)| :=\mu_{0}$ exists. There exist only finitely many disjoint intervals in the set $\{ \Gamma_{k(x)}(x): x\in E_{\lambda}(f) \}$ with length equal to $\mu_{0}$. Let $\Delta_{j} (1\leq j\leq n_{1})$ be all such intervals. Let $E_{2} = E_{\lambda}(f)\setminus(\bigcup_{j=1}^{n_{1}} \Delta_{j})$ and repeat the same procedure taking $E_{2} $ instead of $E_{\lambda}(f).$ Thus step by step we construct the finite or denumerable set of disjoint closed intervals $\{ \Delta_{k} \}_{k\in \Upsilon}$ which satisfy the conditions of lemma.
\end{proof}
\begin{proof}[Proof of Lemma \ref{lem:HGlambda}]
Let $\Delta_{l}$ be an arbitrary closed interval from (\ref{eq:El}).
At the first step we have to prove the following relation
\begin{equation}\label{eq:glam1}
\bigg|\bigg\{x\in \Delta_{l}: G_{m}(f,x)> 2\lambda\quad \text{and}\quad M_{\mathcal M}f(x) \leq \gamma \lambda\bigg\}\bigg| \leq C\gamma  |\Delta_{l}|,
 \end{equation}
 where $C>0$ is independent of $f$,$\lambda$,$\gamma$ and $\Delta_{l}$.
Suppose that there is at least a point  $y_{l}\in \Delta_{l}$ such that
$M_{\mathcal M}f(y_{l}) \leq \gamma \lambda$. Otherwise there is nothing to prove. Let $\Delta^{*}_{l} \in {\mathcal M}$ be the interval which satisfies the following conditions: $\Delta^{*}_{l} \supset \Delta_{l}, |\Delta^{*}_{l}| = m|\Delta_{l}|.$ Let $f(x) = f_{1}(x) + f_{2}(x)$, where
$$
f_{1}(x) = \bigg(f(x)- f_{\Delta^{*}_{l}}\bigg)\chi_{\Delta^{*}_{l}}(x),\qquad f_{\Delta} = \frac{1}{|\Delta|}\int_{\Delta} f(t) dt,\, \Delta\in {\mathcal M}
$$
 and $f_{2}(x) = f(x) - f_{1}(x)$. By Proposition \ref{Gm:1} we have that
\begin{equation}\label{eq:13}
\bigg|\bigg\{ G_{m}(f_{1},x)> \frac{\lambda}{2} \bigg\}\bigg|  \leq \frac{2 C_{1}}{\lambda} \| f_{1}\|_{L^{1}[0,1]} = \frac{4 C_{1}}{\lambda} \int_{\Delta^{*}_{l}} |f(t)| dt
\end{equation}
\[
 \leq \frac{4m C_{1}}{\lambda}|\Delta_{l}| M_{\mathcal M}f(y_{l}) \leq 4m C_{1} \gamma |\Delta_{l}|.
\]
On the other hand we have that $|\Delta^{*}_{l}| = m^{-\kappa}$ for some $\kappa \in \IN_{0}$. Thus for all $0\leq l \leq \mu_{\kappa}$ we have that
$a_{l}(f) = a_{l}(f_{2})$, which yields $G_{m}(f_{2},x) = G_{m}(\Theta_{\mu_{\kappa}}(f,\cdot),x)$ if $x\in \Delta^{*}_{l}$. There exists at least one point $z_{l}\in \Delta^{*}_{l}$ such that $G_{m}(f,z_{l}) \leq \lambda$. Thus if $x\in \Delta^{*}_{l}$ then
\[
G_{m}(f_{2},x) = G_{m}(\Theta_{\mu_{\kappa}}(f,\cdot),x) \leq G_{m}(f,z_{l})\leq \lambda.
\]
We have that if $x\in \Delta^{*}_{l}$ then
\[
G_{m}(f,x) \leq G_{m}(f_{1},x) + G_{m}(f_{2},x) \leq G_{m}(f_{1},x) + \lambda.
\]
 Hence, by (\ref{eq:13}) we finish the proof of (\ref{eq:glam1}).

We have that the weight function $w$  satisfies the condition ${\mathcal M}_{\infty}([0,1]).$  Thus we obtain that
\[
w(\{x\in \Delta_{l}: G_{m}(f,x)> 2\lambda\, \text{and}\, M_{\mathcal M}f(x) \leq \gamma \lambda\}) \leq C\gamma^{\delta}  w(\Delta_{l})
\]
where $C>0$ is independent of $f$, $\lambda >0$,  $\gamma >0$ and $\Delta_{l}$. Hence, by Lemma  \ref{eq:El} we finish the proof.
\end{proof}
\begin{proof}[Proof of Theorem \ref{WHUB:1}]
The necessity follows from Theorem \ref{WHB:1}.  Suppose that $w$ satisfies the condition ${\mathcal M}_{p}([0,1])$.  By Lemma \ref{lem:HGlambda} we derive
\begin{align*}
\int_{[0,1]} G^{p}_{m}(f,x) w(x) dx = & p 2^{p}\int_{0}^{+\infty} \lambda^{p-1}  w(\{x\in [0,1]: G_{m}(f,x)> 2\lambda \}) d\lambda\\  &\leq K_{p} \int_{0}^{+\infty} \lambda^{p-1}  w(\{x\in [0,1]: M_{\mathcal M}f(x) > \gamma \lambda \}) d\lambda \\
&+ K_{p}C \gamma^{\delta} \int_{0}^{+\infty} \lambda^{p-1}  w(\{x\in [0,1]: G_{m}(f,x)> \lambda \}) d\lambda
 \end{align*}
Let $\gamma_{0}>0$ be such that $K_{p}C \gamma_{0}^{\delta} <\frac{1}{2}.$ Then we obtain that
\[
\int_{[0,1]} G^{p}_{m}(f,x) w(x) dx \leq 2 K_{p}\gamma_{0}^{-p} \int_{[0,1]} M^{p}_{\mathcal M}f(x) w(x) dx.
\]
By Proposition \ref{pr:3} we finish the proof.
\end{proof}

\

\section{The system ${\mathcal H}_{0}(m) = \{ {\textsf{h} }_{l}(x) \}_{l=1}^{\infty}$ in $L^{p}([0,1],w)$ }

In this section we will use the following result (see \cite{K:0}--\cite{K:3})
\begin{Theorem}\label{thm:CMS} Let $\{f\sb n \} \sb {n=1} \sp \infty \subseteq L^{\infty}(E)$ be an orthonormal system of real-valued
functions defined on a measurable set $E, 0<|E|< +\infty$ and suppose that $\{f\sb n \} \sb {n=1} \sp \infty$ is  total with
respect to $L^{1}(E) $. Let, furthermore, $N\in \IN$   and $w \in L^{1}(E)$ be a weight function. For the system
$\{f\sb n \} \sb {n=N+1} \sp \infty$ to be closed and/or minimal
it is necessary and sufficient that the following conditions 1)
and/or 2), respectively, are satisfied:
\par
1) any function of the form
$(w)\sp {-1} \sum \sb {n=1} \sp
N c\sb n f\sb n$, where $c\sb n (1 \leq n \leq N )$ are real numbers,
belongs to $L^{p'}(E,w)$ if and only if every $c\sb n$ is zero;
\par
2) for every $k \, (k=N+1,N+2,...)$ there exist uniquely determined
real numbers $b\sb n \sp {(k)} \, (1 \leq n \leq N )$ such that the
function
\[
g\sb k =  \frac{1}{w} \bigg[\sum \sb {n=1} \sp N b\sb n \sp
{(k)} f\sb n + f\sb k \bigg]
\]
 belongs to $L^{p'}(E,w) \quad (\frac{1}{p} + \frac{1}{p'}
= 1)$.
\end{Theorem}
The following two lemmas are easy consequences of Theorem \ref{thm:CMS}. We skip the details of the proofs because they are similar to the case of the Haar system \cite{K:4}.

\begin{Lemma}\label{lem:HmC}
For any $m=2,3,\ldots$ the system ${\mathcal H}_{0}(m)$ is complete in a weighted norm space $L^{p}([0,1],w),$ $ 1\leq p < \infty$ if and only if there exists at least one point $y\in [0,1]$  such that
\begin{equation}\label{eq:C}
\frac{1}{w} \notin  L^{\frac{1}{p-1}}(\Delta_{j}(y))\qquad \text{for all} \quad j\in \IN.
\end{equation}
\end{Lemma}
\begin{Lemma}\label{lem:HmM}
For any $m=2,3,\ldots$ the system ${\mathcal H}_{0}(m)$ is minimal in a weighted norm space $L^{p}([0,1],w),$ $ 1\leq p < \infty$ if and only if
$\frac{1}{w} \in  L^{1}([0,1])$ or for a point $y\in [0,1]$
\begin{equation}\label{eq:M}
\frac{1}{w} \in  L^{\frac{1}{p-1}} ([0,1]\setminus \Delta_{j}(y))\qquad \text{for all} \quad j\in \IN.
\end{equation}
\end{Lemma}
By Lemmas \ref{lem:HmC} and \ref{lem:HmM} it follows easily
\begin{Lemma}\label{lem:HmCM}
For any $m=2,3,\ldots$ the system ${\mathcal H}_{0}(m)$ is complete and minimal in a weighted norm space $L^{p}([0,1],w),$ $ 1\leq p < \infty$ if and only if there exists only one point $y\in [0,1]$  such that the conditions  (\ref{eq:C}), (\ref{eq:M}) hold.
\end{Lemma}
Further in this section we will suppose that the weight function $w$ satisfies the conditions (\ref{eq:C}) and (\ref{eq:M}). Hence, the system ${\mathcal H}_{0}(m)$ is complete and minimal in the weighted norm space $L^{p}([0,1],w),$ $ 1\leq p < \infty$ with the unique conjugate system ${\mathcal H}^{*}_{0}(m)$. By Theorem \ref{thm:CMS} applied for our case it is easy to see that the  system ${\mathcal H}^{*}_{0}(m) = \{ {\textsf{h} }^{*}_{l}(x) \}_{l=1}^{\infty} $ is defined by the following equations:
\begin{equation}\label{eq:dauls}
{\textsf{h} }^{*}_{l}(x) = \frac{{\textsf{h} }_{l}(x) - {\textsf{h} }_{l}(y)}{w(x)}.
\end{equation}

For any $f\in L^{p}([0,1],w)$ and for any $1\leq j\leq m^{k}, k\in \IN$  we put
\begin{equation}\label{eq:psum}
\Theta^{(0)}_{\mu_{k} + j(m-1)}(f,x) =    \sum_{l=1}^{\mu_{k}} c_{l}(f) {\textsf{h} }_{l}(x) +  \sum_{s=0}^{j-1} \sum_{\nu =1}^{m-1} c^{(\nu)}_{k,s,m}(f) h^{(\nu)}_{k,s,m}(x),
\end{equation}
where
\[
  c_{l}(f) =\int_{[0,1]} f(t) {\textsf{h} }^{*}_{l}(t) dt; \quad  c^{(\nu)}_{k,s,m}(f) =\int_{[0,1]} f(t)[ h^{(\nu)}_{k,s,m}(t) - h^{(\nu)}_{k,s,m}(y)]dt.
\]
Let $\Delta_{kj}(y)$ be the interval from the collection of sets (\ref{delta1}), (\ref{delta2}) such that $y\in \Delta_{kj}(y).$

\begin{Lemma}\label{Hylem:1}
For any $f\in L^{p}([0,1],w)$ and for any $1\leq j\leq m^{k}, k\in \IN$ we have that  $\Theta^{(0)}_{\mu_{k} + j(m-1)}(f,x)$ is constant on any interval from the collection of sets (\ref{delta1}), (\ref{delta2}).
Moreover,
\begin{equation}\label{sumpy}
\Theta^{(0)}_{\mu_{k}+ j(m-1) }(f,x) =  -\frac{1}{|\Delta_{kj}(y)|}\int_{[0,1]\setminus\Delta_{kj}(y)} f(t) dt\quad \text{for}\quad x\in \Delta_{kj}(y),
\end{equation}
and for any $\Delta$ from (\ref{delta1}) or from (\ref{delta2}) which does not coincide with $\Delta_{kj}(y)$
\begin{equation}\label{delysu}
\Theta^{(0)}_{\mu_{k}+ j(m-1) }(f,x) =  \frac{1}{|\Delta|}\int_{\Delta} f(t) dt\quad \text{for}\quad x\in \Delta.
\end{equation}
\end{Lemma}

\begin{proof}
In the proof we use the notation of Lemma \ref{Hlem:ker1}. Let $\Delta $ be any interval from the collection of sets (\ref{delta1}), \ref{delta2}) such that $\Delta \bigcap \Delta_{kj}(y)= \emptyset$. Then  we  have that
\[
\Theta^{(0)}_{\mu_{k}+ j(m-1) }(f,x) =  \int_{[0,1]} f(t)[ \tilde{K}_{kj}(t,x) - \tilde{K}_{kj}(y,x)]dt
  \]
  \[ = \sum_{i=1}^{m^{k} + j(m-1)} \int_{G_{i}} f(t)[ K_{kj}(t,x) - K_{kj}(y,x)]dt,
\]
where $\tilde{K}_{kj}(t,x) = K_{kj}(t,x) -1$, hence, $\tilde{K}_{kj}(t,x) - \tilde{K}_{kj}(y,x) = K_{kj}(t,x) - K_{kj}(y,x)$.
Suppose that $G_{\nu} = \Delta_{kj}(y)$ and take any $i_{0} \neq \nu, 1\leq i_{0}\leq m^{k} + j(m-1)$. Then by Lemma \ref{Hlem:ker1} we obtain that for $x\in G_{i_{0}}$ we obtain that
 \[  \sum_{i=1}^{m^{k} + j(m-1)} \int_{G_{i}} f(t)[ K_{kj}(t,x) - K_{kj}(y,x)]dt = \int_{G_{i_{0}}} f(t) K_{kj}(t,x) dt
  \]
  \[= \frac{1}{|G_{i_{0}}|}\int_{G_{i_{0}}} f(t) dt.
\]
On the other hand if $x\in G_{\nu}$  by Lemma \ref{Hlem:ker1} we will have that
\[
\Theta^{(0)}_{\mu_{k}+ j(m-1) }(f,x) = \sum_{i=1, i\neq \nu}^{m^{k} + j(m-1)} \int_{G_{i}} f(t)[ K_{kj}(t,x) - K_{kj}(y,x)]dt
 \]
 \[
= - \frac{1}{|G_{\nu}|} \sum_{i=1, i\neq \nu}^{m^{k} + j(m-1)} \int_{G_{i}} f(t)dt  = -\frac{1}{|G_{\nu}|}\int_{[0,1]\setminus G_{\nu}} f(t) dt.
\]

\end{proof}

Lemma \ref{lem:wsing1} and Lemma \ref{lem:HmCM} easily yield
\begin{Lemma}\label{lem:wsing2}
Let $w\geq 0$ be a weight function defined on $[0,1]$ such that $w$ satisfies the condition ${\mathcal M}^{y}_{p}([0,1])$ for some  $y\in [0,1]$ and $ 1 < p < \infty$.
Then the conditions  (\ref{eq:C}), (\ref{eq:M}) hold and
 the system ${\mathcal H}_{0}(m)$ is complete and minimal in a weighted norm space $L^{p}([0,1],w).$
\end{Lemma}

\begin{Theorem}\label{WHB0:1}
For any $m=2,3,\ldots$ the system ${\mathcal H}_{0}(m)$ is a basis in the weighted norm space $L^{p}([0,1],w),$ $ 1< p < \infty$ if and only if there exists a point $y\in [0,1]$ such that   $w$ satisfies the conditions ${\mathcal M}_{p}([0,1]\setminus \{y\})$ and  ${\mathcal M}^{y}_{p}([0,1])$.
\end{Theorem}
\begin{proof}
Necessity. If the system ${\mathcal H}_{0}(m)$ is a basis in the weighted norm space $L^{p}([0,1],w)$ then it is a complete minimal system in $L^{p}([0,1],w)$.
Then we will have that  for any $f\in L^{p}([0,1],w)$ and for any $1\leq j\leq m^{k}, k\in \IN$ the partial sum operators are uniformly bounded
\[
\sup_{1\leq j\leq m^{k}, k\in \IN}\| \Theta^{(0)}_{\mu_{k}+ j(m-1) }\|_{L^{p}([0,1],w) \to L^{p}([0,1],w)}:= M_{0} \leq B_{p},
\]
where $B_{p} >0$. We have that
\[
M_{0} \geq  \max_{1\leq i \leq m^{k} + j(m-1)}\, \, \sup_{\| f\|_{L^{p}(G_{i},w)}\leq 1} \| \Theta^{(0)}_{\mu_{k}+ j(m-1) }\|_{L^{p}([0,1],w)}
\]
If $G_{i}\bigcap \Delta_{kj}(y)= \emptyset$ then
 by Lemma \ref{Hylem:1} we will have that
 \[
 \sup_{\| f\|_{L^{p}(G_{i},w)}\leq 1} \| \Theta^{(0)}_{\mu_{k}+ j(m-1) }\|_{L^{p}(G_{i},w)} = \frac{1}{|G_{i}|}\sup_{\| f\|_{L^{p}(G_{i},w)}\leq 1} \bigg|\int_{G_{i}} f(t) dt \bigg| w(G_{i})^{\frac{1}{p}}
 \]
 \[
= |G_{i}|^{-1} w(G_{i})^{\frac{1}{p}} \| w^{-\frac{1}{p}}\|_{L^{p'}(G_{i})}.
 \]
If $G_{\nu} = \Delta_{kj}(y)$ then in the same way as above we obtain that
\[
|\Delta_{kj}(y)|^{-p} w(\Delta_{kj}(y)) \bigg(\int_{\Delta_{kj}(y)} w^{-\frac{1}{p-1}} dt\bigg)^{p-1} \leq B_{p}^{p}.
 \]
 Sufficiency. By Lemma \ref{lem:wsing2} we have that the system ${\mathcal H}_{0}(m)$ is complete and minimal in a weighted norm space $L^{p}([0,1],w).$ Hence, by Lemma \ref{Hylem:1} we obtain that for any $f\in L^{p}([0,1],w)$ and for any $1\leq j\leq m^{k}, k\in \IN$
 \[
 \int_{[0,1]}|\Theta^{(0)}_{\mu_{k}+ j(m-1) }(f,t)|^{p} w(t) dt
 = \sum_{\overset{1\leq i\leq m^{k} + j(m-1)}{i\neq \nu}} \bigg|\frac{1}{|G_{i}|}\int_{G_{i}} f(t) dt\bigg|^{p} \int_{G_{i}} w(t) dt
 \]
 \[
 + \bigg|\frac{1}{|G_{\nu}|}\int_{[0,1]\setminus G_{\nu}} f(t) dt\bigg|^{p} \int_{G_{\nu}} w(t) dt \leq B_{p}^{p} \int_{[0,1]}|f(t)|^{p} w(t) dt.
 \]
The last inequality follows  because $w$ satisfies the conditions ${\mathcal M}_{p}([0,1]\setminus \{y\})$ and  ${\mathcal M}^{y}_{p}([0,1])$.
To finish the proof we have to show that \newline $\lim_{l\to \infty} |c_{l}(f)| \|{\textsf{h} }_{l}(\cdot)\|_{L^{p}([0,1],w)} = 0.$ We skip the details because a similar result we have proved for the proof of Theorem \ref{WHB:1}.

\end{proof}
In the case $p=1$ we have the following result.
\begin{Theorem}\label{WHB0:2}
For any $m=2,3,\ldots$ the system ${\mathcal H}_{0}(m)$ is a basis in the weighted norm space $L^{1}([0,1],w),$  if and only if $\frac{1}{w} \notin L^{\infty}([0,1])$ and there exists a point $y\in [0,1]$ such that   $w$ satisfies the conditions ${\mathcal M}_{1}([0,1]\setminus \{y\})$ and  ${\mathcal M}^{y}_{1}([0,1])$.
\end{Theorem}
We will not give the details of the proof because it is similar to the proof of Theorem \ref{WHB0:1}.
The main theorem of this section is the following
\begin{Theorem}\label{UNB0:1}
If the system ${\mathcal H}_{0}(m), m=2,3,\ldots$ is a basis in the weighted norm space $L^{p}([0,1],w),$ $ 1< p < \infty$ then ${\mathcal H}_{0}(m)$  is an unconditional basis in the same space.
\end{Theorem}
\begin{proof}
By Theorem \ref{WHB0:1} we have that there exists $y\in [0,1]$ such that the weight function $w$ satisfies the conditions ${\mathcal M}_{p}([0,1]\setminus \{y\})$ and  ${\mathcal M}^{y}_{p}([0,1])$. For any $f\in L^{p}([0,1],w)$  there exists a unique sequence $\{a_{l}(f)\}_{l=1}^{\infty}$ such that
\begin{equation}\label{eq:f1}
f= \sum_{l=1}^{\infty} a_{l}(f) {\textsf{h} }_{l}.
\end{equation}
The coefficients which correspond to the functions $h^{(\nu)}_{\Delta_{j}(y)}$ in the series (\ref{eq:f1}) we denote by $b_{j\nu}$.
We split formally the series (\ref{eq:f1}) into two parts
\begin{equation}\label{eq:f2}
\sum_{l=1}^{\infty} a_{l}(f) {\textsf{h} }_{l} = \sum_{j=0}^{\infty} \sum_{\nu=1}^{m-1} b_{j\nu}(f) h^{(\nu)}_{\Delta_{j}(y)} + \sum{}^{'} a_{l}(f) {\textsf{h} }_{l},
\end{equation}
where by $\sum'$ we have denoted the series obtained after excluding the terms which are present in the first series.

 For any $k\in \IN_{0}$ let $G_{kl} \subset {\mathcal M}, 1\leq l \leq m-1$ be mutually disjoint intervals such that $|G_{kl}| = m^{-k-1},  1\leq l \leq m-1$ and
  $\Delta_{k}(y) = \Delta_{k+1}(y)\bigcup \bigcup_{l=1}^{m-1} G_{kl}$.

 By Theorem \ref{WHUB:1} we easily obtain that the series $\sum{}^{'} a_{l}(f) {\textsf{h} }_{l}$ converges unconditionally in $L^{p}(G_{kl},w)$ for any $k\in \IN$ and for all $1\leq l\leq m-1$. Hence, if we check that the series $\sum{}^{'} a_{l}(f) {\textsf{h} }_{l}$ converges in $L^{p}([0,1],w)$ we will have that it converges unconditionally in $L^{p}([0,1],w)$. Thus the proof of theorem will be finished if we prove that the first series on the right hand side of the equality (\ref{eq:f2}) converges unconditionally in $L^{p}([0,1],w)$. Recall that  we are using the notation introduced in (\ref{haar:m2}). Let
\begin{equation}\label{eq:f3}
F(x) = \sum_{j=0}^{\infty} \sum_{\nu=1}^{m-1} b_{j\nu}(f) h^{(\nu)}_{\Delta_{j}(y)}(x) = d^{l}_{k}\qquad \text{if}\quad x\in G_{kl}.
\end{equation}
for all $k\in \IN_{0}$ and $1\leq l\leq m-1$. By Lemma \ref{Hylem:1} we have that
\[
d^{l}_{k} = \frac{1}{|G_{kl}|} \int_{G_{kl}} f(t) dt, \qquad \text{for}\quad x\in G_{kl}.
\]
The weight function $w$ satisfies the condition ${\mathcal M}_{p}([0,1]\setminus \{y\})$. Hence,
\[
\int_{[0,1]} |F(x)|^{p} w(x) dx = \sum_{k=0}^{\infty} \sum_{l=1}^{m-1} |d^{l}_{k}|^{p} w(G_{kl})
\]
\[
\leq  \sum_{k=0}^{\infty} \sum_{l=1}^{m-1} |G_{kl}|^{-p} \int_{G_{kl}} |f(t)|^{p} w(t) dt \bigg(\int_{G_{kl}}  w(t)^{-\frac{1}{p-1}} dt\bigg)^{p-1} w(G_{kl})
\]
\[
\leq  C_{p} \int_{[0,1]} |f(t)|^{p} w(t) dt.
\]
The system ${\mathcal H}_{0}(m)$ is a basis in the weighted norm space $L^{p}([0,1],w).$  Hence, the first series in the right hand side of the equality (\ref{eq:f2}) converges in $L^{p}([0,1],w).$ Thus to finish the proof of Theorem \ref{UNB0:1} we have to prove that the series in (\ref{eq:f3}) converge unconditionally in $L^{p}([0,1],w).$

For any $j\in \IN_{0}$ we have that
\[
\sum_{\nu=1}^{m-1} b_{j\nu}(f) h^{(\nu)}_{\Delta_{j}(y)}(x) = d^{l}_{j} = \frac{1}{|G_{jl}|} \int_{G_{jl}} f(t) dt,
\]
 for $x\in G_{jl}, 1\leq l\leq m-1$ and
\[
\sum_{\nu=1}^{m-1} b_{j\nu}(f) h^{(\nu)}_{\Delta_{j}(y)}(x): = -c_{j}= -\sum_{l=1}^{m-1}d^{l}_{j}, \quad \text{for}\quad x\in \Delta_{j+1}(y).
\]
Let $\{ \gamma_{l}\}_{l=0}^{m-1}$ be a collection of numbers such that
\begin{equation}\label{eq:num}
\sum_{l=0}^{m-1} \gamma^{2}_{l} =1 \quad \text{and }\quad \sum_{l=0}^{m-1} \gamma_{l} =0.
\end{equation}
We put
\begin{equation}\label{eq:ksi}
\xi_{j}(x) = |\Delta_{j+1}(y)|^{-\frac{1}{2}}[\gamma_{0} \chi_{\Delta_{j+1}(y)}(x) + \sum_{l=1}^{m-1} \gamma_{l} \chi_{G_{jl}}(x)]
\end{equation}
and
\begin{equation}\label{eq:alpha}
\alpha_{j}(f) = \int_{[0,1]}f(t) [\xi_{j}(t) - |\Delta_{j+1}(y)|^{-\frac{1}{2}}\gamma_{0}] dt
\end{equation}
\[
 = - \frac{\gamma_{0}}{\sqrt{|\Delta_{j+1}(y)|}} \int_{[0,1]\setminus{\Delta_{j+1}(y)}} f(t) dt +  \int_{[0,1]\setminus{\Delta_{j+1}(y)}} f(t) \xi_{j}(t) dt
\]
\[
= |\Delta_{j+1}(y)|^{ \frac{1}{2}}\, \sum_{l=1}^{m-1} (\gamma_{l} - \gamma_{0}) d^{l}_{j} - \frac{\gamma_{0}}{\sqrt{|\Delta_{j+1}(y)|}} \int_{[0,1]\setminus{\Delta_{j}(y)}} f(t) dt.
\]
\[
= |\Delta_{j+1}(y)|^{ \frac{1}{2}}\, \bigg(\sum_{l=1}^{m-1} (\gamma_{l} - \gamma_{0}) d^{l}_{j} - \gamma_{0} \sum_{s=0}^{j-1} m^{j-s-1} c_{s}\bigg).
\]

\begin{Lemma}\label{lem:fval}
  For any $\varepsilon = \{ \epsilon_{j}\}_{j=0}^{\infty}$ let
\begin{equation}\label{eq:numfu}
F^{*}_{\varepsilon} (x) = \sum_{j=0}^{\infty}\epsilon_{j} \alpha_{j} \xi_{j}(x).
\end{equation}
Then for all $k\in \IN$ and $x\in \Delta_{k}(y) \setminus \Delta_{k+1}(y)$
\[
|F_{\varepsilon}^{*} (x)| \leq   2\sum_{s=0}^{k} \sum_{l=1}^{m-1}  |d^{l}_{s}| + \frac{1}{m-1} \sum_{s=0}^{k-1} m^{k-s} |c_{s}|.
\]
\end{Lemma}
\begin{proof}
By (\ref{eq:ksi}) and (\ref{eq:alpha}) we obtain that
 for $x\in G_{k\nu}, 1\leq \nu\leq m-1$
\[
|F_{\varepsilon}^{*} (x)| \leq  \sum_{j=0}^{k} |\alpha_{j} \xi_{j}(x)|  =  |\gamma_{0}|\sum_{s=0}^{k-1} \sum_{l=1}^{m-1} |\gamma_{l} - \gamma_{0}| |d^{l}_{s}| + |\gamma_{\nu}|\sum_{l=1}^{m-1} |\gamma_{l}- \gamma_{0}| |d^{l}_{k}|
\]
\[
+ \sum_{j=1}^{k-1} \sum_{s=0}^{j-1} m^{j-s-1} |c_{s}| \leq 2\sum_{s=0}^{k} \sum_{l=1}^{m-1}  |d^{l}_{s}| + \sum_{j=1}^{k-1} \sum_{s=0}^{j-1} m^{j-s-1} |c_{s}|
\]
\[
 \leq 2\sum_{s=0}^{k} \sum_{l=1}^{m-1}  |d^{l}_{s}| + \frac{1}{m-1} \sum_{s=0}^{k-1} m^{k-s} |c_{s}|.
\]
\end{proof}

\begin{Lemma}\label{lem:fva1}
 For any $f \in L^{p}([0,1],w),$ $ 1< p < \infty$ and
 any $\varepsilon = \{ \epsilon_{j}\}_{j=0}^{\infty}$  the function $F_{\varepsilon}^{*} \in L^{p}([0,1],w)$ and
 \[
 \|F_{\varepsilon}^{*} \|_{L^{p}([0,1],w)} \leq C'_{p} \|f\|_{L^{p}([0,1],w)},
 \]
 where $C'_{p}>0$ is independent of $f$ and $\varepsilon$.
\end{Lemma}
\begin{proof}
By Lemma \ref{lem:fval} we have that
\[
\int_{\Delta_{k}(y) \setminus \Delta_{k+1}(y)} |F^{*} (x)|^{p} w(t) dt \leq
4^{p}\bigg(\sum_{s=0}^{k}\sum_{l=1}^{m-1}  |d^{l}_{s}| \bigg)^{p} w(\Delta_{k}(y))
\]
\[
+ \frac{2^{p}}{(m-1)^{p}}\bigg( [w(\Delta_{k}(y))]^{\frac{1}{p}}  \sum_{s=0}^{k-1} m^{k-s} |c_{s}|\bigg)^{p}.
\]
Afterwards write
\[
|\Delta_{s+1}(y)|\sum_{l=1}^{m-1}  |d^{l}_{s}|\, [w(\Delta_{k}(y))]^{\frac{1}{p}}
\]
\[
\leq   \bigg(\int_{\Delta_{s}(y)\setminus \Delta_{s+1}(y)} |f(t)|^{p}w(t) dt\bigg)^{\frac{1}{p}} \bigg(\int_{\Delta_{s}(y)\setminus \Delta_{s+1}(y)} w(t)^{-\frac{1}{p-1}} dt\bigg)^{\frac{1}{p'}}[w(\Delta_{k}(y))]^{\frac{1}{p}} .
\]
\[
\leq  C^{\frac{1}{p}}_{p} \bigg(\int_{\Delta_{s}(y)\setminus \Delta_{s+1}(y)} |f(t)|^{p}w(t) dt\bigg)^{\frac{1}{p}} |\Delta_{k}(y)|.
\]
Hence, we obtain that
\[
\sum_{s=0}^{k} \sum_{l=1}^{m-1}  |d^{l}_{s}|\, [w(\Delta_{k}(y))]^{\frac{1}{p}} \leq C^{\frac{1}{p}}_{p} \bigg(\int_{\Delta_{s}(y)\setminus \Delta_{s+1}(y)} |f(t)|^{p}w(t) dt\bigg)^{\frac{1}{p}} \frac{1}{m^{k-s}}
\]
Now we apply  the following lemma which is a consequence of Theorem 274 from \cite{HLP:522}.
\begin{Lemma}\label{lem:HLP}
Let $u=\{u_{j}\}_{j=0}^{\infty}$ and  $v=\{v_{j}\}_{j=0}^{\infty}$ be  numerical sequences such that  $u\in l^{1}$ and $v\in l^{p}, p>1$. Then  the Cauchy product $w=\{w_{n}\}_{n=0}^{\infty},$ $w_{n}= \sum_{j=0}^{n} u_{n-j}v_{j}$ of the sequences $u$ and $v$ belongs to $l^{p}$. Moreover \newline
$\|w\|_{l^{p}} \leq \|u\|_{l^{1}} \|v\|_{l^{p}}$.
\end{Lemma}
Which gives us the convergence of the series
\[
\sum_{k=1}^{\infty} \bigg(\sum_{s=0}^{k} \sum_{l=1}^{m-1}  |d^{l}_{s}|\bigg)^{p}\, w(\Delta_{k}(y)) \leq 2^{p}C_{p} \int_{[0,1]} |f(t)|^{p}w(t) dt.
\]
To finish the proof of Lemma \ref{lem:fva1} we have to show that
\begin{equation}\label{eq:conv}
\sum_{k=1}^{\infty} \bigg( [w(\Delta_{k}(y))]^{\frac{1}{p}}  \sum_{s=0}^{k-1} m^{k-s} |c_{s}|\bigg)^{p} < +\infty
\end{equation}
We have that
\[
(\sum_{s=0}^{k-1} m^{-s-1} |c_{s}|)^{p}
\]
\[
\leq  \bigg(\sum_{s=0}^{k-1} \sum_{l=1}^{m-1} \bigg(\int_{G_{sl}} |f(t)|^{p}w(t) dt\bigg)^{\frac{1}{p}} \bigg(\int_{G_{sl}}  w(t)^{-\frac{1}{p-1}} dt\bigg)^{\frac{1}{p'}}\bigg)^{p}
\]
\[
\leq  \bigg(\sum_{s=0}^{k-1}  \bigg(\int_{\Delta_{s}(y)\setminus \Delta_{s+1}(y)} |f(t)|^{p}w(t) dt\bigg)^{\frac{1}{p}}
\]
\[
\times \bigg(\int_{\Delta_{s}(y)\setminus \Delta_{s+1}(y)} w(t)^{-\frac{1}{p-1}} dt\bigg)^{\frac{1}{p'}} \bigg)^{p}.
\]
Recall that $w$ satisfies ${\mathcal M}^{y}_{p}([0,1])$. By Lemma \ref{lem:wsing1} we obtain that
\[
m^{k} \bigg(\int_{\Delta_{s}(y)\setminus \Delta_{s+1}(y)} w(t)^{-\frac{1}{p-1}} dt\bigg)^{\frac{1}{p'}} w(G_{kl})^{\frac{1}{p}}
\]
\[
\leq C^{\frac{1}{p}}_{p}  \bigg(\int_{\Delta_{s}(y)\setminus \Delta_{s+1}(y)} w(t)^{-\frac{1}{p-1}} dt\bigg)^{\frac{1}{p'}} \bigg[\int_{[0,1]\setminus \Delta_{k}(y)} \omega^{-\frac{1}{p-1}}(t) dt\bigg]^{-\frac{1}{p'}}
\]
\[
\leq C^{\frac{1}{p}}_{p} q_{p}^{-\frac{k-s}{p}}.
\]
If we write
\[
\sum_{s=0}^{k-1} m^{k-s-1} |c_{s}| w(G_{kl})^{\frac{1}{p}}
 \]
 \[
 \leq C^{\frac{1}{p}}_{p} \sum_{s=0}^{k-1}  \bigg(\int_{\Delta_{s}(y)\setminus \Delta_{s+1}(y)} |f(t)|^{p}w(t) dt\bigg)^{\frac{1}{p}} q_{p}^{-\frac{k-s}{p}}
\]
and put
$v_{j} = (\int_{\Delta_{j}(y)\setminus \Delta_{j+1}(y)} |f(t)|^{p}w(t) dt )^{\frac{1}{p}}$, $u_{j}= q_{p}^{-\frac{j}{p}}$ then by Lemma \ref{lem:HLP}
we will obtain
\[
\sum_{k=1}^{\infty} \bigg(\sum_{s=0}^{k-1} m^{k-s-1} |c_{s}|\bigg)^{p} w(G_{kl}) \leq C_{p} B_{p} \int_{[0,1]} |f(t)|^{p}w(t) dt.
\]

\end{proof}

Lemma \ref{lem:fva1} yields  the convergence  of  the series
\[
\sum_{j=0}^{\infty} \epsilon^{(\nu)}_{j} b_{j\nu}(f) h^{(\nu)}_{\Delta_{j}(y)}
\]
 for any $1\leq \nu \leq m-1$ and  $ \forall \varepsilon^{(\nu)} = \{ \epsilon^{(\nu)}_{j}\}_{j=0}^{\infty},$ where $\epsilon^{(\nu)}_{j}= \pm 1$. Moreover, we obtain that for some $B_{p}>0$
\[
\bigg\| \sum_{j=0}^{\infty} \sum_{\nu=1}^{m-1} \epsilon^{(\nu)}_{j} b_{j\nu}(f) h^{(\nu)}_{\Delta_{j}(y)} \bigg\|_{L^{p}([0,1],w)} \leq B_{p} \| f \|_{L^{p}([0,1],w)}.
\]
\end{proof}

\section{Higher rank Haar wavelets in $L^{p}(\mathbb{R},\omega)$}

Let $\omega \geq 0$ be a locally integrable function defined on $\IR$.
In this section we  study the phenomenon  described in the introduction with respect to the higher rank Haar wavelet systems $H(m), m=2,3,\dots.$
Let $\chi^{-}(x) = \chi_{\IR^{-}}(x)$ and $\chi^{+}(x) = \chi_{\IR^{+}}(x)$.
The following result is the first step in that direction.
\begin{Lemma} \label{lem:phen1}
For any $m=2,3,\dots$
let $H(m)$ be the wavelet system  defined by (\ref{haar:m}) and  (\ref{haar:m1}).
Let $U_{m}$ be the linear subspace of locally integrable functions $\xi$ on $\IR$ such that
\begin{equation}\label{eq:orth}
 \int_{\mathbb{R}}
\xi(t) h^{(\nu)}_{k,j,m}(x)(t)dt=0 \qquad  \forall j,k\in\mathbb{Z}, 1\leq \nu \leq m-1.
\end{equation}
Then $\dim U_{m} =2$ and $\chi^{-}$, $\chi^{+}$ as vectors constitute a basis in $U_{m}$.
\end{Lemma}
\begin{proof}
It is clear that if we prove that a locally integrable function $\xi$ such that $\xi(x) =0$ if $x\in \IR^{-}$ and holds (\ref{eq:orth}) if and only if $\xi = c\chi^{+}$ for some $c\in  \IR$ then the proof will be finished.
By Corollary \ref{Htot:1} we have that the system  ${\mathcal H}(m),$ $m=2,3,\ldots$ is total with respect to $ L^{1}[0,1]$. Hence, by definition of  the system ${\mathcal H}(m)$ and by (\ref{eq:orth}) it follows that
\[
\int_{[0.1]} \xi(t){\textsf{h} }_{l}(t) dt = 0\qquad \text{for all} \quad l\in \IN.
\]
Which yields that $\xi(x)= c\, {\textsf{h} }_{0}(x)$ for $x\in [0,1]$. We finish the proof by induction.
Suppose that for some $N\in \IN$ it is true that if $\xi$ is a locally integrable function such that $\xi(x) =0$ if $x\in \IR^{-}$ and (\ref{eq:orth}) is true then $\xi(x) =c_{0}$ if $x\in [0,m^{N}],$ where $c_{0}\in \IR$. If $\xi$ is a function which satisfies to all mentioned conditions then by definition of the system $H(m)$ it follows that the functions $\xi_{\nu}(x) = \xi(x-\nu m^{N}), 1\leq \nu \leq m-1$. Thus by our supposition it follows that
$\xi_{\nu}(x) = c_{\nu}$ if $x\in [0,m^{N}],$ where $c_{\nu}\in \IR$. Hence, $\xi(x) =c_{\nu}$ if $x\in [\nu m^{N}, (\nu+1)m^{N}], 0\leq \nu\leq m-1.$
Afterwards we observe that the functions $h^{(\nu)}(m^{N+1}x), 1\leq \nu\leq m-1 $ belong to the system $H(m)$, which yields
\[
\int_{0}^{m^{N+1}} \xi(x) h^{(\nu)}(m^{N+1}x) dx = 0 \qquad \text{for all}\quad 1\leq \nu\leq m-1.
\]
After a change of the variable we have that
\[
\int_{0}^{1} \xi(m^{-N-1} t) h^{(\nu)}(t) dt = 0 \qquad \text{for all}\quad 1\leq \nu\leq m-1.
\]
By definition of the functions $h^{(\nu)}, 1\leq \nu\leq m-1$ we obtain that $\xi(m^{-N-1} x) = c$ if $x\in [0,1].$
\end{proof}
It is convenient to continue our study considering the systems $H^{+}(m),$ $H^{-}(m)$ respectively in the spaces $L^{p}(\IR^{+},\omega)$ and
$L^{p}(\IR^{-},\omega)$. It is easy to see some sort of symmetry between those systems. Thus it would be sufficient to study the system $H^{+}(m)$ in the space $L^{p}(\IR^{+})$. In fact we have proved the analogue of the above lemma for the system $H^{+}(m)$ which is formulated as follows.
\begin{Lemma} \label{lem:phpl}
Let $U^{+}_{m}$ be the linear subspace of locally integrable functions $\xi$ on $\IR^{+}$ such that
\begin{equation}\label{eq:or1}
 \int_{\mathbb{R}^{+}}
\xi(t) h^{(\nu)}_{k,j,m}(x)(t)dt=0 \qquad  \forall k\in\mathbb{Z},\forall j\in \mathbb{Z}^{+}, 1\leq \nu \leq m-1.
\end{equation}
Then $\dim U^{+}_{m} =1$ and  $\chi^{+} \in U^{+}_{m}$.
\end{Lemma}

We  need the analogues of Lemmas \ref{lem:HmC}, \ref{lem:HmM} for this case.

\begin{Lemma} \label{lem:50}
The  system $H^{+}(m)$  is complete in $L^p(\IR^{+},\omega),$ $ 1\leq p <\infty$ if and only if
\begin{equation}\label{eq:cc}
\frac{\chi^{+}}{\omega} \notin L^{\frac{1}{p-1}}(\IR^{+}).
\end{equation}
\end{Lemma}
\begin{proof}
Suppose
that $H^{+}(m)$  is  complete in $L^p(\IR^{+},\omega).$ If
 $g= \frac{\chi^{+}}{\omega} \in L^{\frac{1}{p-1}}(\IR^{+})$ then
$  g\in L^{p'}(\IR^{+},\omega),$ where $\frac{1}{p} +\frac{1}{p'} =1.$
Thus
\begin{equation}\label{eq:compl}
\int_{\mathbb{R}^{+}}g(t)h^{(\nu)}_{k,j,m}(x)(t)\omega(t)dt=0 \qquad  \forall k\in\mathbb{Z},\forall j\in \mathbb{Z}^{+}, 1\leq \nu \leq m-1.
\end{equation}
Which yields that $H^{+}(m)$  is not complete in $L^p(\IR^{+},\omega)$. Which is a contradiction.

Suppose that $ \frac{\chi^{+}}{\omega} \notin L^{\frac{1}{p-1}}(\IR^{+})$. If $H^{+}(m)$  is not complete in $L^p(\IR^{+},\omega)$ then there exists
$g\in L^{p'}(\IR^{+},\omega)$ such that (\ref{eq:compl}) holds. By Lemma \ref{lem:50} it follows that $g(t) \omega(t) = c \chi^{+}(t)$ a.e. on $\IR^{+}$, where $c\in \IR$. We came to a contradiction which finishes the proof.
\end{proof}
From Lemma \ref{lem:50} follows
\begin{Lemma}\label{lem:HWmC}
The  system $H^{+}(m)$  is complete in $L^p(\IR^{+},\omega),$ $1\leq p <\infty$  if and only if there exists at least one point $y\in [0,+\infty]$  such that
\begin{equation}\label{eq:CW}
\frac{1}{w} \notin  L^{\frac{1}{p-1}}(\Delta_{j}(y))\qquad \text{for all} \quad j\in \IN.
\end{equation}
\end{Lemma}
We also have
\begin{Lemma} \label{lem:51}
 The  system $H^{+}(m)$  is minimal in $L^p(\IR^{+},\omega),$ $1\leq p <\infty$ if and only if
 \begin{equation*} \makebox{\parbox{4.1in}{ For any $h^{(\nu)}_{k,j,m}(x) \in H^{+}(m)$  there exists a coefficient  $a^{(\nu)}_{k,j,m}$ such that
$$
g^{(\nu)}_{k,j,m} = \frac{a^{(\nu)}_{k,j,m} \chi^{+} +  h^{(\nu)}_{k,j,m} }{\omega} \in L^{\frac{1}{p-1}}(\IR^{+}).
$$}}\tag{M}
\end{equation*}
  \end{Lemma}
\begin{proof}
Suppose that the  system $H^{+}(m)$  is minimal in $L^p(\IR^{+},\omega)$. Then there exists a system $\{ g^{(\nu)}_{k,j,m}: k\in \IZ, j\in \IZ^{+},1\leq \nu \leq m-1\}$ biorthogonal to $H^{+}(m)$.
Hence, if for some $\nu_{0}, 1\leq \nu_{0} \leq m-1$ we fix any $l\in \IZ$ and any $\mu\in \IZ^{+}$ then  for all $k\in \IZ, j\in \IZ^{+}$ and $ 1\leq \nu \leq m-1$
\begin{equation}\label{eq:bior}
\int_{\mathbb{R}^{+}} [ g^{(\nu_{0})}_{l,\mu,m}(x)
\omega(x) - h^{(\nu_{0})}_{l,\mu,m}(x)] h^{(\nu)}_{k,j,m}(x) dx=0.    
\end{equation}
By Lemma \ref{lem:phpl} we obtain that
\[
g^{(\nu_{0})}_{l,\mu,m} = \frac{a^{(\nu_{0})}_{l,\mu,m} \chi^{+} +  h^{(\nu_{0})}_{l,\mu,m} }{\omega} \in L^{\frac{1}{p-1}}(\IR^{+}).
\]
The proof of  sufficiency is direct. We easily check that the system \newline $\{ g^{(\nu)}_{k,j,m}: k\in \IZ, j\in \IZ^{+},1\leq \nu \leq m-1\}$ is biorthogonal to $H^{+}(m)$.
\end{proof}
From Lemma \ref{lem:51} easily follows
\begin{Lemma}\label{lem:HWmM}
The  system $H^{+}(m)$  is minimal in $L^p(\IR^{+},\omega),$ $1\leq p <\infty$  if and only if there exists at most one point $y\in [0,+\infty]$  such that (\ref{eq:CW}) holds.
\end{Lemma}
By Lemmas \ref{lem:50} and \ref{lem:51} we obtain immediately
\begin{Lemma} \label{lem:52}
The  system $H^{+}(m)$  is complete and  minimal in $L^p(\IR^{+},\omega),$ $1\leq p <\infty$ if and only if conditions (\ref{eq:cc})
and (M) hold.
\end{Lemma}
Lemmas \ref{lem:HWmC} and \ref{lem:HWmM} yield
\begin{Lemma}\label{lem:HWmCM}
The  system $H^{+}(m)$  is complete and minimal in $L^p(\IR^{+},\omega),$ $1\leq p <\infty$  if and only if there exists a unique point $y\in [0,+\infty]$  such that the condition (\ref{eq:CW}) holds.
\end{Lemma}
If we analyze the proofs of results which brought us the last lemma then it is not hard to see that the following result also holds.
\begin{Lemma}\label{lem:HW-CM}
The  system $H^{-}(m)$  is complete minimal in $L^p(\IR^{-},\omega),$ $1\leq p <\infty$  if and only if there exists a unique point $y\in [-\infty,0]$  such that the condition (\ref{eq:CW}) holds.
\end{Lemma}
Lemma \ref{lem:HWmCM} and Lemma \ref{lem:HW-CM} easily yield
\begin{Lemma}\label{lem:HWRCM}
The  system $H(m)$  is complete and minimal in $L^p(\IR,\omega),$ $1\leq p <\infty$  if and only if there exists a unique point $y^{+}\in [0,+\infty]$ and a unique point $y^{-}\in [-\infty,0]$ such that the condition (\ref{eq:CW}) holds for both of those points.
\end{Lemma}
The following lemma will be used in the proof of the main result of the present section.
\begin{Lemma}\label{lem:ker}
Let $\{\phi_{j}\}_{j=1}^{\mu}, \{\psi_{j}\}_{j=1}^{\mu}$ be some measurable functions defined on a measurable set $E$ and let
\[
K(x,t) = \sum_{j=1}^{\mu} \phi_{j}(x) \psi_{j}(t) \qquad (x,t)\in E\times E.
\]
Furthermore, for a given real valued orthogonal matrix
\[
A = \bigg( a_{ij} \bigg)_{\substack{1\leq i\leq \mu \\1\leq j\leq \mu}}
\]
let $f_{k}(x) = \sum_{i=1}^{\mu} a_{ik} \phi_{i}(x)$, $g_{k}(x) = \sum_{i=1}^{\mu} a_{ik} \psi_{i}(t).$ If we consider a new kernel
 $\Phi(x,t) = \sum_{k=1}^{\mu} f_{k}(x) g_{k}(t)$ then
\[
K(x,t) = \Phi(x,t) \qquad \text{for}\quad (x,t)\in E\times E.
\]
\end{Lemma}
\begin{proof}
We have
\begin{align*}
\Phi(x,t) &= \sum_{k=1}^{\mu} \sum_{i=1}^{\mu} a_{ik} \phi_{i}(x) \sum_{\nu=1}^{\mu} a_{\nu k} \psi_{\nu}(t)\\
 &=  \sum_{i=1}^{\mu} \sum_{\nu=1}^{\mu} \phi_{i}(x) \psi_{\nu}(t)  \sum_{k=1}^{\mu} a_{\nu k} a_{ik} \\
 &=   \sum_{\nu=1}^{\mu} \phi_{\nu}(x) \psi_{\nu}(t) = K(x,t).
\end{align*}
\end{proof}
We are going to apply Lemma \ref{lem:ker} in the proof of the next theorem. As $\{\phi_{j}\}_{j=1}^{\mu}$ and $\{f_{k}\}_{k=1}^{\mu}$ we will take two orthonormal bases in $V(m)$ considered in Section \ref{sec:HW}. Concretely we will consider the following orthonormal bases of $V(m):$ $\{h^{(\nu)}(x): 0\leq \nu \leq m-1\}$ and $\{\varphi_{1,j,m}(x): 0\leq j\leq m-1\}$.
\begin{Theorem}\label{WHUB:1pl}
For any $m=2,3,\ldots$ the system $H^{+}(m)$ is an unconditional basis in the weighted norm space $L^{p}(\IR^{+},\omega),$ $ 1< p < \infty$ if and only if there exists a point $y\in [0,+\infty]$ such that: \newline
If $y\neq +\infty$ then $\omega$ satisfies the condition ${\mathcal M}_{p}(\IR^{+}\setminus \{y\})$ and the condition ${\mathcal M}^{y}_{p}(\IR^{+})$;  \newline
If $y=+\infty$   then   $\omega$ satisfies the condition ${\mathcal M}_{p}(\IR^{+}).$
\end{Theorem}
\begin{proof}
Suppose that $H^{+}(m)$ is an unconditional basis in the weighted norm space $L^{p}(\IR^{+},\omega),$ $ 1< p < \infty$. Then $H^{+}(m)$  is complete minimal in $L^p(\IR^{+},\omega)$  and  there exists a unique point $y\in [0,+\infty]$  such that the condition (\ref{eq:CW}) holds. First consider the case $y=+\infty.$
In this case the uniqueness of the point $y$ means that for any  $h^{(\nu)}_{k,j,m}(x) \in H^{+}(m)$
\begin{equation}\label{eq:dual1}
g^{(\nu)}_{k,j,m} = \frac{ h^{(\nu)}_{k,j,m} }{\omega} \in L^{\frac{1}{p-1}}(\IR^{+}).
\end{equation}
  The proof of the necessity  can be easily completed following the scheme  of the proof of Theorem \ref{WHB:1}.\newline
If $y\in [0,+\infty)$ then by Lemma \ref{lem:51} the system biorthogonal to $H^{+}(m)$ is defined by the following equations:
\begin{equation}\label{eq:dual2}
g^{(\nu)}_{k,j,m}(t) = \frac{ h^{(\nu)}_{k,j,m}(t) -  h^{(\nu)}_{k,j,m}(y) \chi^{+}(t) }{\omega(t)}
\end{equation}
for all $k\in \IZ, j\in \IZ^{+}$ and $ 1\leq \nu \leq m-1.$ \newline
Let $\Delta = \Delta_{l+1}(y)$ and let $\Delta_{l}(y) = \Delta_{l}(0)+ j_{y}m^{-l},$  where $j_{y}\in \IN_{0}$. For $f\in L^{p}(\IR^{+},w),$ consider the sum
\begin{align*}
&\sum_{\nu=1}^{m-1} c^{(\nu)}_{l,j_{y},m}(f)\, h^{(\nu)}_{l,j_{y},m}(x)\\ &= \sum_{\nu=0}^{m-1}\int_{\IR^{+}} f(t) g^{(\nu)}_{k,j_{y},m}(t) \omega(t) dt\, h^{(\nu)}_{l,j_{y},m}(x) - c^{(0)}_{l,j_{y},m}(f)\, h^{(0)}_{l,j_{y},m}(x)\\ & = \int_{\IR^{+}} f(t) \sum_{\nu=0}^{m-1}g^{(\nu)}_{k,j_{y},m}(t) h^{(\nu)}_{l,j_{y},m}(x) \omega(t) dt - c^{(0)}_{l,j_{y},m}(f)\, h^{(0)}_{l,j_{y},m}(x),
\end{align*}
where $h^{(0)}_{l,j_{y},m}(x) = \chi_{\Delta_{l}(y)}(x)$ and
\[
g^{(0)}_{k,j,m}(t) = \frac{ h^{(0)}_{k,j,m}(t) -  h^{(0)}_{k,j,m}(y) \chi^{+}(t) }{\omega(t)}.
\]
By Lemma \ref{lem:ker} follows that
\begin{align}\label{eq:y1}
&\int_{\IR^{+}} f(t) \sum_{\nu=0}^{m-1}g^{(\nu)}_{k,j_{y},m}(t) h^{(\nu)}_{l,j_{y},m}(x) \omega(t) dt \\ &= \int_{\IR^{+}} f(t) \sum_{j=0}^{m-1}\varphi_{l+1,j_{y}+j,m}(x) \psi_{l,j,m}(t)\omega(t) dt,
\end{align}
where
\[
\psi_{l,j,m}(t) = \frac{\varphi_{l+1,j_{y}+j,m}(t) -  \varphi_{l+1,j_{y}+j,m}(y) \chi^{+}(t) }{\omega(t)}.
\]
Thus we obtain that if $x\in \Delta_{l+1}(y)$ then
\[
\sum_{\nu=1}^{m-1} c^{(\nu)}_{l,j_{y},m}(f)\, h^{(\nu)}_{l,j_{y},m}(x) = - m^{l+1} \int_{\IR^{+}\setminus \Delta_{l+1}(y)}  f(t)dt - m^{l} \int_{\IR^{+}\setminus \Delta_{l}(y)}  f(t) dt.
\]
If $f(t) \geq 0$ for $ t\in \IR^{+}$ then it follows that
for $x\in \Delta_{l+1}(y)$
\[
|\sum_{\nu=1}^{m-1} c^{(\nu)}_{l,j_{y},m}(f)\, h^{(\nu)}_{l,j_{y},m}(x)| \geq  |\Delta_{l+1}(y)|^{-1} |\int_{\IR^{+}\setminus \Delta_{l+1}(y)}  f(t)dt|.
\]
Afterwards in the same way as in the proof of Theorem \ref{WHB0:1} we obtain that for some $B_{p}>0$ and for all $l\in \IZ$
\[
|\Delta_{l}(y)|^{-p} \omega(\Delta_{l}(y)) \bigg(\int_{\IR^{+}\setminus \Delta_{l}(y)} \omega^{-\frac{1}{p-1}} dt\bigg)^{p-1} \leq B_{p}^{p}.
 \]
Let $E_{lj} \in {\mathcal M}, 1\leq j \leq m-1$ be mutually disjoint intervals such that $|E_{lj}| = m^{-1}|\Delta_{l}(y)|,  1\leq l \leq m-1$ and
  $\Delta_{l}(y) = \Delta_{l+1}(y)\bigcup \bigcup_{j=1}^{m-1} E_{lj}$. By (\ref{eq:y1}) we obtain that if $f(t) = 0$ when $x\in \IR^{+}\setminus \Delta_{l}(y)$ then
\[
\sum_{\nu=1}^{m-1} c^{(\nu)}_{l,j_{y},m}(f)\, h^{(\nu)}_{l,j_{y},m}(x) = |E_{lj}|^{-1} \int_{E_{lj}}  f(t)dt \qquad \text{if}\quad x\in E_{lj} (1\leq j\leq m-1).
\]
Which yields
\[
|E_{lj}|^{-p} \omega(E_{lj}) \bigg(\int_{E_{lj}} \omega^{-\frac{1}{p-1}} dt\bigg)^{p-1} \leq B_{p}^{p} \quad \text{for any}\, l\in \IZ\quad \text{and}\, 1\leq j\leq m-1.
 \]
Let $\Delta \in {\mathcal M}, |\Delta| = m^{-l-1}$ be such that $\Delta \bigcap \Delta_{l}(y)= \emptyset.$ We consider the interval $\Delta^{*} \in {\mathcal M}, |\Delta^{*}| = m^{-l}$ be such that $\Delta \subset \Delta^{*}.$ Let  $ = \Delta_{l}(0)+ k^{*}m^{-l},$  where $k^{*}\in \IN_{0}$.

For $f\in L^{p}(\IR^{+},w)$ consider the sum
\[
\sum_{\nu=1}^{m-1} c^{(\nu)}_{l,k^{*},m}(f)\, h^{(\nu)}_{l,k^{*},m}(x).
\]
Using the same idea as above we show that if $x\in \Delta$ then
\[
\sum_{\nu=1}^{m-1} c^{(\nu)}_{l,k^{*},m}(f)\, h^{(\nu)}_{l,k^{*},m}(x) = |\Delta|^{-1} \int_{ \Delta}  f(t)dt - |\Delta^{*}|^{-1} \int_{ \Delta^{*}}  f(t)dt.
\]
Thus if $f(t) =0$ for $t\in \IR^{+}\setminus \Delta$ it follows that
\[
\sum_{\nu=1}^{m-1} c^{(\nu)}_{l,k^{*},m}(f)\, h^{(\nu)}_{l,k^{*},m}(x) = \frac{m-1}{m}|\Delta|^{-1} \int_{ \Delta}  f(t)dt
\]
and the proof of the necessity is completed easily.

The proof of the sufficiency will be given following the same idea as in proof of Theorem \ref{WUnB:1}. By Lemma \ref{lem:HWmCM} we have that the system $H^{+}(m)$ is complete and minimal in $ L^{p}(\IR^{+},\omega)$. Let $G^{+}(m) = \{g^{(\nu)}_{k,j,m}: k\in \IZ, j\in \IZ^{+}, 1\leq \nu \leq m-1\}$ be the conjugate system of the basis $H^{+}(m)$. Suppose that  $\omega$ satisfies the condition ${\mathcal M}_{p}(\IR^{+}).$ Then the system $G^{+}(m)$ is defined by the equations (\ref{eq:dual1}).
 Let $f\in L^{p}(\IR^{+},\omega)$ and
let $\Omega \subset \IZ\times \IZ^{+}$ be a finite set. Moreover, let $N\in \IN$ be such that $h^{(\nu)}_{k,j,m}(x) = 0$ if $x\in [m^{N}, +\infty)$ for all $(k,j)\in \Omega$.
For any $1\leq \nu \leq m-1$ consider the sum
\begin{equation}\label{eq:sum}
S^{(\nu)}_{\Omega}(f,x) = \sum_{(k,j)\in \Omega} c^{(\nu)}_{kj}(f) h^{(\nu)}_{k,j,m}(x),
\end{equation}
where
\[
c^{(\nu)}_{kj}(f) = \int_{\IR^{+}} f(t) h^{(\nu)}_{k,j,m}(t) dt.
\]
Applying Lemma \ref{lem:transf1}  as in the proof of Theorem \ref{WUnB:1} we obtain that for some
\[
\| S^{(\nu)}_{\Omega}(f,\cdot)\|_{L^{p}(\IR^{+},\omega)} \leq B_{p} \| f\|_{L^{p}(\IR^{+},\omega)},
\]
where $B_{p}>0$ is independent of $f$ and $\Omega$.\newline
If $y\in [0,+\infty)$ then we take $N\in \IN$ so that $h^{(\nu)}_{k,j,m}(x) = 0$ if $x\in [m^{N}, +\infty)$ for all $(k,j)\in \Omega$ and $y\in [0,m^{N}]$.
By Lemma \ref{lem:51} the system $G^{+}(m)$ is defined by the equations (\ref{eq:dual2}).
In this case the coefficients of the sum (\ref{eq:sum}) are defined as follows:
\[
c^{(\nu)}_{kj}(f) = \int_{\IR^{+}} f(t)[ h^{(\nu)}_{k,j,m}(t) - h^{(\nu)}_{k,j,m}(y) ]dt.
\]
We write $f(t) = f_{1}(t) + f_{2}(t)$, where $f_{1}(t) = f(t) \chi_{[0,m^{N}]}(t)$.
By Lemma \ref{lem:transf2} and Lemma \ref{lem:transf3} in the same way as above we obtain that
\[
\| S^{(\nu)}_{\Omega}(f_{1},\cdot)\|_{L^{p}(\IR^{+},\omega)} \leq B_{p} \| f_{1}\|_{L^{p}(\IR^{+},\omega)},
\]
where $B_{p}>0$ is independent of $f_{1}$ and $\Omega$. The proof will be complete if we show that
\[
\| S^{(\nu)}_{\Omega}(f_{2},\cdot)\|_{L^{p}(\IR^{+},\omega)} \leq B^{*}_{p} \| f_{2}\|_{L^{p}(\IR^{+},\omega)},
\]
where $B^{*}_{p}>0$ is independent of $f_{2}$ and $\Omega$. We have that
\[
S^{(\nu)}_{\Omega}(f_{2},x) = \sum_{(k,j)\in \Omega} c^{(\nu)}_{kj}(f_{2}) h^{(\nu)}_{k,j,m}(x)
\]
\[
= -\int_{[m^{N},+\infty)} f(t) dt \sum_{(k,j)\in \Omega}h^{(\nu)}_{k,j,m}(y) h^{(\nu)}_{k,j,m}(x).
\]
Let $f^{*}_{N}(t) = m^{-N} \chi_{[0,m^{N}]}(t)$ then
\begin{align*}
\sum_{(k,j)\in \Omega}h^{(\nu)}_{k,j,m}(y) h^{(\nu)}_{k,j,m}(x) &= \sum_{(k,j)\in \Omega} \int_{\IR^{+}} f^{*}_{N}(t)[ h^{(\nu)}_{k,j,m}(t) - h^{(\nu)}_{k,j,m}(y)]dt \, h^{(\nu)}_{k,j,m}(x) \\
 &= S^{(\nu)}_{\Omega}(f^{*}_{N},x).
\end{align*}
Thus we have that
\begin{align*}
&\| S^{(\nu)}_{\Omega}(f_{2},\cdot)\|_{L^{p}(\IR^{+},\omega)} = \bigg|\int_{[m^{N},+\infty)} f(t) dt\bigg| \| S^{(\nu)}_{\Omega}(f^{*}_{N},\cdot)\|_{L^{p}(\IR^{+},\omega)}\\
& \leq \| f_{2} \|_{L^{p}(\IR^{+},\omega)} \bigg[\int_{[m^{N},+\infty)} \omega(t)^{-\frac{1}{p-1}} dt\bigg]^{\frac{1}{p'}}
  B_{p}\| f^{*}_{N} \|_{L^{p}(\IR^{+},\omega)}\\
&= B_{p} m^{-N} \bigg[\int_{[m^{N},+\infty)} \omega(t)^{-\frac{1}{p-1}} dt\bigg]^{\frac{1}{p'}}
 \bigg[\int_{[0,m^{N}]} \omega(t) dt\bigg]^{\frac{1}{p}} \| f_{2} \|_{L^{p}(\IR^{+},\omega)}.
\end{align*}
Using that $\omega$ satisfies the condition ${\mathcal M}^{y}_{p}(\IR^{+})$ we complete the proof.
\end{proof}
It is easy to check that any function $\omega_{r}(x) = x^{r}$ if $r>p-1$ satisfies the condition ${\mathcal M}^{0}_{p}(\IR^{+})$ and the condition
${\mathcal M}_{p}((0,+\infty))$.
\begin{Corollary}
Let $ 1< p < \infty$ and let $\omega_{r}(x) = |x|^{r}$ if $r>p-1$.  Then the system $H^{+}(m)$ is an unconditional basis in the weighted norm space $L^{p}(\IR^{+},\omega_{r})$.
\end{Corollary}

\begin{Theorem}\label{WHUB:2pl}
For any $m=2,3,\ldots$ the system $H(m)$ is an unconditional basis in the weighted norm space $L^{p}(\IR,\omega),$ $ 1< p < \infty$ if and only if there exist  two points  $y_{1}\in [0,+\infty], y_{2}\in [-\infty,0]$ such that: \newline
If $y_{1}\neq +\infty$ then $\omega$ satisfies the condition ${\mathcal M}_{p}(\IR^{+}\setminus \{y_{1}\})$ and the condition ${\mathcal M}^{y_{1}}_{p}(\IR^{+})$;  \newline
If $y_{1}=+\infty$   then   $\omega$ satisfies the condition ${\mathcal M}_{p}(\IR^{+});$\newline
If $y_{2}\neq +\infty$ then $\omega$ satisfies the condition ${\mathcal M}_{p}(\IR^{-}\setminus \{y_{2}\})$ and the condition ${\mathcal M}^{y_{2}}_{p}(\IR^{-})$;  \newline
If $y_{2}=-\infty$   then   $\omega$ satisfies the condition ${\mathcal M}_{p}(\IR^{-});$
\end{Theorem}

\begin{Corollary}
Let $ 1< p < \infty$ and let $\omega_{r}(x) = |x|^{r}$ if $r>p-1$.  Then the system $H(m)$ is an unconditional basis in the weighted norm space $L^{p}(\IR,\omega_{r})$.
\end{Corollary}

\end{document}